\documentclass[12pt]{amsart}
\usepackage[margin=1.5in]{geometry}
\usepackage{graphicx} 
\usepackage{amsfonts,amsmath,amsthm,amssymb}
\usepackage{mathtools}
\usepackage[aligntableaux=bottom,boxsize = 0.8cm]{ytableau}
\usepackage{tikz,varwidth}
\usepackage{float}

\usepackage[nocompress]{cite}
\usetikzlibrary{calc}

\newtheorem{theorem}{Theorem}[section]
\newtheorem{proposition}[theorem]{Proposition}
\newtheorem{lemma}[theorem]{Lemma}
\newtheorem{corollary}[theorem]{Corollary}

\theoremstyle{definition}
\newtheorem{remark}[theorem]{Remark}
\newtheorem{example}[theorem]{Example}
\newtheorem{definition}[theorem]{Definition}

\newtheorem{maintheorem}{Theorem}

\DeclareMathOperator{\conv}{conv}

\newcommand{\boxs}{0.8cm}
\newcommand{\labA}{\operatorname{lab}^A}

\newcommand{\labL}{\operatorname{lab}^{L}}

\newcommand{\RR}{\mathbb{R}}
\newcommand{\FF}{\mathbb{F}}

\title{Lattice Path Delta Matroids}
\author{Douglas M. Chen \and Mario Sanchez \and John Veliz \and Zhiyan Ying}

\address{Johns Hopkins University}
\email{dchen101@jhu.edu}
\address{Cornell University}
\email{ms2962@cornell.edu}
\address{Cornell University}
\email{jv424@cornell.edu}
\address{University of California, Los Angeles}
\email{kying0607@ucla.edu}

\date{\today}

\usepackage{hyperref}
\hypersetup{
    colorlinks,
}

\begin{document}

\begin{abstract}
    We initiate the study of a type $C_n$ generalization of the lattice path matroids defined by Bonin, de Mier, and Noy. These are delta matroids whose feasible sets are in bijection with lattice paths which are symmetric along the main diagonal. We describe deletion, contraction, enveloping matroids, and the homogeneous components of these delta matroids in terms of the combinatorics of symmetric lattice paths.
    
    In the second half, we study the convex geometry of the feasible polytopes of lattice path delta matroids. We show that these polytopes decompose into the feasible polytope of special lattice paths corresponding to intervals in the type $C_n$ Gale order whose related Richardson variety is a toric variety. Further, these special polytopes inherit a unimodular triangulation from Stanley's Eulerian triangulation of the hypercube. As a consequence, the volume of these polytopes is the proportion of permutations with specific ascent sets.
\end{abstract}
\maketitle

\tableofcontents
\section{Introduction}

Lattice path matroids are a special class of matroids first introduced by Bonin, de Mier, and Noy \cite{Bonin2003}. These matroids are parameterized by a pair of lattice paths $p$ and $q$ consisting of only eastward and northward unit steps from $(0,0)$ to a point $(a,b)$ such that $p$ is weakly below $q$. The bases of a lattice path matroid $M[p,q]$ are in bijection with the lattice paths weakly above $p$ and weakly below $q$. Their main utility is that many of their properties can be described in terms of the combinatorics of lattice paths. This includes deletion, contraction, duality, activity, calculations of Tutte polynomials, volumes of their matroid polytopes, integer lattice point enumerations, facial structures, and many more \cite{Bonin2003,Bonin2006,Bonin2010,Bid12,knauer2018lattice}.

In the general trend of Coxeter/Lie combinatorics, many commonly studied combinatorial objects are related to the type $A_n$ root system, representation theory, or corresponding geometry. Through this connection, one finds natural analogues of the usual combinatorial objects for arbitrary root systems. For matroids, these analogues are known as Coxeter matroids \cite{Borovik2003}. Among this general class of objects, the most studied are the Lagrangian matroids, which are the generalization corresponding to the type $C_n$ cominuscule parabolic quotient. Lagrangian matroids have been previously studied in optimization under the name of delta matroids.

The goal of this paper is to study the delta matroid generalization of lattice path delta matroids. To do this, we define a \textbf{symmetric lattice path} as a lattice path from $(0,0)$ to $(n,n)$ that is symmetric along the line $y = n - x$. We say $p \leq q$ if $p$ is always weakly below $q$. A \textbf{lattice path delta matroid} of a pair of symmetric lattice paths $p \leq q$ is a delta matroid whose feasible sets are in bijection with symmetric lattice paths $r$ with $p \leq r \leq q$. The bijection sends each lattice path $p$ to a certain subset of labels $\labL(p)$ in analogy with the ordinary case. Under this bijection, lattice path delta matroids are indexed by pairs of subsets $S,T \subseteq [n]$ such that $S \leq T$ in the type $C_n$ Gale order. 

We begin by defining lattice path delta matroids in Section \ref{sec: Lattice Path Model} and constructing an example which is a type $C_n$ version of Ardila's Catalan matroid \cite{Ardila2003}. We will also discuss the connection with algebraic geometry through representability and moment map images of Richardson varieties in the Lagrangian Grassmannian.

In Section \ref{Sec: First Properties}, we study first properties and operations of lattice path delta matroids. We give concrete descriptions of the deletion, contraction, and duality of lattice path delta matroids in terms of the subsets which imply the following:

\begin{proposition}
    The set of lattice path delta matroids is closed under delta matroid deletion, contraction, and duality.
\end{proposition}

We also study various ways of obtaining ordinary matroids from delta matroids. In particular, we describe the homogeneous components of lattice path delta matroids in terms of the corresponding symmetric skew-diagram and obtain natural enveloping matroids for lattice path delta matroids $\Delta[p,q]$ in the sense of \cite{EFLS22}.

In the second half of the paper, we focus on the convex geometry of delta matroids. For a delta matroid $\Delta$, the \textbf{feasible polytope} of $\Delta$ is
 \[P(\Delta) = \operatorname{conv}(e_B \; \lvert \; B \in \mathcal{F}(\Delta)) \subseteq \RR^n. \]

In Section \ref{Sec: Polytopes}, we study the dimension, intersections, and faces of these polytopes. The main result of the section is the following hyperplane description of $P(\Delta[S,T])$:
\begin{maintheorem}\label{mainthm: hyperplane description}
    Let $S$ and $T$ be subsets of $[n]$ such that $S \leq T$ in the Gale order of type $C_n$. Then, $x \in P(\Delta[S,T])$ if and only if
     \[|S_{\geq i}| \leq x_i + \cdots + x_n \leq |T_{\geq i}| \]
     and
     \[0 \leq x_i \leq 1\]
    for all $i \in [n]$, where $S_{\geq i }\coloneqq \{x \in S \; \lvert \; x \geq i\}$.
\end{maintheorem}

In Section \ref{Sec: Snake Paths and Triangulations}, we study a special class of lattice path delta matroids $\Delta[p,q]$ where the skew-diagram bounded by $p$ and $q$ does not contain a $2 \times 2$ box. We refer to these intervals as \textbf{symmetric snake paths}. We say that an interval $[p,q]$ is \textbf{linked} if $p$ and $q$ only meet at $(0,0)$ and $(n,n)$.

Algebro-geometrically, we show that these are exactly the lattice path delta matroids whose corresponding Richardson variety in $\operatorname{LGr}_{n,2n}(\FF)$ is a toric variety. Combinatorially, we give a full description in terms of subsets for the linked symmetric snake paths.

\begin{proposition}
    Let $[p,q]$ be a linked symmetric snake path. Then, $\Delta[p,q] = \Delta[S, S \cup \{n\}]$ for some subset $S \subseteq [n-1]$ and every such subset appears in this way.
\end{proposition}

Our main result on symmetric snake paths is that their feasible polytopes admit a natural unimodular triangulation coming from Stanley's Eulerian triangulation of the hypercube \cite{Sta77}. Stanley's triangulation consists of $n!$ simplices naturally indexed by permutations $w \in S_n$ with the special property that the simplices indexed by permutations with $k$ ascents give a triangulation for the hypersimplex $\Delta_{k,n} \subseteq [0,1]^n$. We prove an extension of this result when we fix a particular ascent set. Recall that $i$ is an ascent of $\pi$ if $\pi(i) < \pi(i+1)$.

\begin{maintheorem}\label{mainthm: snake path triangulation}
    The feasible polytope of a lattice path delta matroid $\Delta[p,q] = \Delta[S,S \cup \{n\}]$ of a linked symmetric snake path $[p,q]$ has a unimodular triangulation where the maximal simplices are in bijection with permutations with ascent set $S$.
\end{maintheorem}

\begin{corollary}
    The volume of $\Delta[p,q] = \Delta[S, S \cup \{n\}]$ for a linked symmetric snake path is $\frac{1}{n!}$ times the number of permutations with ascent set $S$.
\end{corollary}

In Section \ref{Sec: Decompositions LPDM}, we relate general lattice path delta matroids to the ones coming from symmetric snake paths. A \textbf{polytope subdivision} of a polytope $P$ is a collection of polytopes $P_i$ such that $\cup P_i = P$ and the intersection of any two $P_i$ and $P_j$ is a face of both polytopes. The theory of subdivisions, and the related notion of valuations, of matroid polytopes and Coxeter matroid polytopes is a rich subject with many important results; see \cite{Ardila2023,Eur2021,ferroni2022valuative}. We give a subdivision of lattice path delta matroids that we hope will find use in the theory of valuations for Lagrangian matroids.

\begin{maintheorem}\label{mainthm: LPDM subdivision into snake paths}
    The feasible polytope of a linked lattice path delta matroid $\Delta[p,q]$ has a polytope decomposition into $P_1, \ldots, P_k$, where each polytope $P_i$ is a lattice path delta matroid of a linked symmetric snake path.
\end{maintheorem}

As a corollary, we calculate the volume of an arbitrary linked lattice path delta matroid by first subdividing into symmetric snake paths and then triangulating each component.

\begin{corollary}
    Let $\Delta[S,T]$ be a linked lattice path delta matroid. Let $m$ be the number of permutations with ascent set $R$ such that $S \leq R \leq T \backslash \{n\}$ in the type $C_n$ Gale order. Then,
     \[\operatorname{vol}(P(\Delta[S,T])) = \frac{m}{n!}. \]
\end{corollary}

\subsection*{Acknowledgements}

This paper was a result of Cornell's Summer Program for Undergraduate Research (SPUR). We thank the organizers for their hospitality. We are also grateful to Allen Knutson and Alexander Vidinas for many helpful conversations. MS was partially funded by the NSF Postdoctoral Research Fellowship DMS-2103136.

\section{Lattice Path Delta Matroids}\label{sec: Lattice Path Model}

Let $[n]$ denote the set $\{1, \ldots, n\}$. We use $e_1, \ldots, e_n$ to denote the standard basis of $\RR^n$ and $e_S = \sum_{i \in S} e_i$ for any subset $S \subseteq [n]$. Given a finite set $E$, we let $\RR^E$ be the vector space with basis $\{e_i\}_{i \in E}$. Let $[\pm n] = \{-n,-(n-1), \ldots, -1,1,2,\ldots,n\}$. For clarity, we often denote $-i$ by $\overline{i}$ in our examples.

We freely use standard definitions and constructions of matroids. See \cite{Oxley} for a general reference. For a matroid $M$ on ground set $[n]$, let $\mathcal{B}(M)$ denote its bases. The \textbf{base polytope} of a matroid $M$ is the polytope
 \[P(M) = \operatorname{conv}(e_B \in \RR^n\; \lvert \; B \in \mathcal{B}(M)). \]

\subsection{Delta Matroids and Lagrangian Matroids}

\begin{definition}
    A \textbf{delta matroid} $\Delta$ on ground set $E$ is a collection of subsets $\mathcal{F}(\Delta)$ of $E$, called feasible sets, such that for any two $A_1,A_2 \in \mathcal{F}(\Delta)$ and any $e \in (A_1 \backslash A_2) \cup (A_2 \backslash A_1)$ there exists an $f \in (A_1 \backslash A_2) \cup (A_2 \backslash A_1)$ such that $(A_1 \backslash \{e,f\}) \cup (\{e,f\} \backslash A_1)$ is in $\mathcal{F}(\Delta)$. 
\end{definition}

We say that a subset $S$ of $[\pm n]$ is \textbf{admissible} if $|S \cap \{i, -i\}| \leq 1$. There is a bijection between admissible subsets of $[\pm n]$ of cardinality $n$ and all subsets of $[n]$ by the map $A \mapsto A \cap [n]$. A \textbf{Lagrangian matroid}\footnote{This is not the standard definition, but it is equivalent. See \cite{Borovik2003} for more information.} $\mathcal{L}$ is a collection $\mathcal{B}(\mathcal{L})$ of admissible subsets of $[ \pm n ]$ of cardinality $n$ such that the corresponding collection of subsets of $[n]$ under the map $S \mapsto S \cap [n]$ is a delta matroid.

For any poset $\mathbb{P}$ and $x \leq y$, let $[x,y] = \{z \in \mathbb{P}\; \lvert \; x \leq z \leq y\}$ denote the interval from $x$ to $y$. A \textbf{chain} is a collection of elements $x_1 \leq \cdots \leq x_k$. A \textbf{maximal chain} is a chain of maximal length. The poset $\mathbb{P}$ is \textbf{ranked} with rank function $\operatorname{rk}$ if every maximal chain has the same length and $\operatorname{rk}(x)$ is the length of a maximal chain ending at $x$. A \textbf{poset isomorphism} is a bijection $\phi: \mathbb{P} \to \mathbb{P}'$ such that $\phi(x) \leq \phi(y)$ if and only if $x \leq y$.

The \textbf{type $C_n$ Gale order} on $2^{[n]}$ is the partial order given by $A \leq B$ for subsets $A,B \subseteq [n]$, where $A = \{a_1 < \cdots < a_j\}$ and $B = \{b_1 < \cdots < b_k\}$, if and only if $j \leq k$ and $a_{j-i+1} \leq b_{k-i+1}$ for all $i \in [j]$. 

Similarly, the \textbf{type $C_n$ Gale order} on admissible subsets of cardinality $n$ of $[\pm n]$ is the partial order given by $A \leq B$, where $A = \{a_1 < a_2 \cdots < a_n \}$ and $B = \{b_1 < b_2 < \cdots < b_n\}$, if and only if $a_i \leq b_i$ for all $i \in [n]$.

For any subset $S \subseteq [n]$, let $S_{\geq i} = \{j \in S \; \lvert \; i \leq j\}$. We will often use the following observation:
\begin{lemma}\label{lem: Gale order and truncated sets}
    Let $S,T \subseteq [n]$. Then, $S \leq T$ in the type $C_n$ Gale order if and only if
    $|S_{\geq i}|\leq |T_{\geq i}|$ for all $i \in [n]$.
\end{lemma}
\begin{proof}
    Assume first that $S=\{ s_1<\cdots<s_j \} \leq T=\{ t_1<\cdots<t_k \}$ in the type $C_n$ Gale order. 
    Then for any $i \in [n]$, writing $S_{\geq i}=\{ s_{i_1}<\cdots<s_{i_\ell} \}$, we have by definition that $s_{i_m}\leq t_{k-j+i_m}$ for all $m \in [\ell]$, which implies $|S_{\geq i}|\leq |T_{\geq i}|$ (this is immediate if $S_{\geq i}=\emptyset$). 
    Conversely, assume that $|S_{\geq i}|\leq |T_{\geq i}|$ for all $i \in [n]$.
    In particular, $j=|S|=|S_{\geq 1}|\leq |T_{\geq 1}|\leq |T|=k$.
    Suppose that $v=s_{j-i+1}>t_{k-i+1}$ for some $i \in [j]$.
    Then $S_{\geq v}=\{ s_{j-i+1}<\cdots<s_j \}$, implying $|S_{\geq v}|=i$, while $T_{\geq v} \subseteq \{ t_{k-i+2}<\cdots<t_k \}$, so that $|T_{\geq v}|\leq i-1$.
    Hence, $|S_{\geq v}|>|T_{\geq v}|$, a contradiction.
    Thus, we must have $s_{j-i+1}\leq t_{k-i+1}$ for all $i \in [j]$.
\end{proof}

\begin{remark}
     By Exercise 8.5 of \cite{BB05}, this order is isomorphic to the induced Bruhat order on the maximal parabolic quotient of $W$ of type $C_n$ by the cominuscule weight.
\end{remark}

The type $C_n$ Gale order is a ranked poset with rank function $\ell$ given by the following lemma:

\begin{lemma}\label{lem: rank calculation of Gale order on subsets}
    The rank of $S$ in the type $C_n$ Gale order on subsets of $[n]$ is $\ell(S) = \sum_{i \in S} i$.
\end{lemma}

\begin{proof}
    We can show this by induction on the cardinality of $S$. When $S = \{i \}$ for some $i \in [n]$, then we have that $\ell(S) = i$ since we have the chain
     \[ \emptyset \leq 1 \leq 2 \leq \cdots \leq i .\]
    In general, if $i$ is the smallest element of $S$, then $\ell(S) = i + \ell(S \backslash i)$ since we have the maximal chain
     \[ S \leq 1 \cup S \leq 2 \cup S \leq \cdots \leq i \cup S. \]
\end{proof}

\begin{proposition}
    Let $S$ and $T$ be subsets of $[n]$, and let $[S,T] = \{A \; \lvert \; S \leq A \leq T\}$ be the interval in the type $C_n$ Gale order. Then $[S,T]$ are the feasible sets of a delta matroid.
\end{proposition}

\begin{proof}
    Tsukerman and Williams showed that the projection of intervals in a Coxeter group $W$ to any parabolic quotient $W/W_I$ are Coxeter matroids. Lagrangian matroids are in bijection with the Coxeter matroids corresponding to the maximal parabolic quotient of the type $C_n$ Coxeter group corresponding to the cominuscule weight. Under this bijection, intervals in the Bruhat order of $W/W_I$ map to intervals in the type $C_n$ Gale order.
\end{proof}

\begin{definition}
    Given a delta matroid $\Delta$ or the corresponding Lagrangian matroid $\mathcal{L}$, the \textbf{feasible polytope} is the polytope
     \[ P(\Delta) = P(\mathcal{L}) = \operatorname{conv}\left (e_S \in \RR^E \; \lvert \; S \in \mathcal{F}(\Delta) \right ).\]
\end{definition}

\begin{theorem}\cite{Borovik2003}\label{thm: type B GGMS}
    A polytope $P$ is the feasible polytope of a delta matroid if and only if $P$ is contained in the cube $[0,1]^n$ and all the edge directions are $e_i - e_j$ or $e_i$.
\end{theorem}

\subsection{Symmetric Lattice Path Model}

We begin by describing lattice paths and the theory of (ordinary) lattice path matroids.

A $(a,b)$-\textbf{lattice path} is a path from $(0,0)$ to $(a,b)$ using only eastward steps $(0,1)$ and northward steps $(1,0)$. We will denote a lattice path by a word in the letters $E$ and $N$ where the letters represent the eastward and northward steps of the path, respectively. We define a partial order on all $(a,b)$-lattice paths by $p \leq q$ if the path $p$ is always below the path $q$. We will often use the following standard description of this partial order in terms of words:

\begin{lemma}
    Let $p = \alpha_1 \cdots \alpha_{a+b}$ and $q = \beta_1 \cdots \beta_{a+b}$ be $(a,b)$-lattice paths. Then, $p$ is weakly below $q$ if and only if the number of eastward steps in the subword $\alpha_k \cdots \alpha_{a+b}$ is smaller than the number of eastward steps in the subword $\beta_k \cdots \beta_{a+b}$ for all $k \in [a+b]$.
\end{lemma}

Given an $(a,b)$-lattice path $p = \alpha_1\ldots\alpha_{a+b}$, the type $A$ \textbf{label} of $p$ is the subset $\labA(p) \subseteq [n]$ given by $i \in \labA$ if and only if $ \alpha_i$ is an eastward step. In fact, the function $\labA$ induces a bijection between $(a,b)$-lattice paths and subsets of $[a+b]$ of cardinality $b$. Combinatorially, this corresponds to labeling every step of the lattice path from $1$ to $a + b$ and then recording the labels of the eastward steps.

The \textbf{type $A$ Gale order} on subsets $\binom{[n]}{k}$ is the partial order given by $A \leq B$ whenever $a_i \leq b_i$ for all $i$, where $A = \{a_1 < a_2 < \cdots a_k\}$ and $B = \{b_1 < b_2 < \cdots < b_k\}$. The bijection $\labA$ gives a poset isomorphism between the set of $(a,b)$ lattice paths with the lattice path partial order and the (type $A$) Gale order on $\binom{[n]}{b}$.

\begin{definition}
    Let $p$ and $q$ be two $(a,b)$-lattice paths corresponding to the subsets $S,T \subseteq [a+b]$ of cardinality $a$ with $p \leq q$ and $S \leq T$. Then, the \textbf{lattice path matroid} $M[p,q]$ or $M[S,T]$ is the rank $a$ matroid on ground set $[a+b]$ whose bases are the elements in the interval $[S,T]$.
\end{definition}

The utility of this theory is that the combinatorics of lattice paths can be used to describe properties of the corresponding matroid. See for instance \cite{Bonin2003, Bonin2006, Bonin2010, BK22}. We now describe a lattice model for the delta matroids corresponding to intervals in the type $C_n$ Gale order.

\begin{definition}
    A \textbf{$n$-symmetric lattice path} $p$ is an $(n,n)$-lattice path that is invariant under the reflection through the line $y = n - x$. In other words, $p = \alpha_1 \ldots \alpha_{2n}$ satisfies $\alpha_i \not = \alpha_{2n - i + 1}$ for all $i \leq n$.
\end{definition}

The \textbf{label} $\labL(p)$ of an $n$-symmetric lattice path $p = \alpha_1 \ldots \alpha_{2n}$ is the admissible subset of $[\pm n]$ given by 
 \[ i \in \labL(p) \iff \begin{cases}
     \alpha_{i + n} = E & \text{for $i \in [n]$} \\
     \alpha_{i + n +1} = E & \text{for $i \in [-n]$}.
 \end{cases}\]
 
 Combinatorially, we obtain $\labL(p)$ by labeling the steps of $p$ in order by the set $[\pm n]$ and then recording the eastward steps.

\begin{proposition}\label{prop: lagrangian paths are admissible subsets}
    The function $\labL$ is a poset isomorphism between $n$-symmetric lattice paths with $p \leq q$ if $p$ is always below $q$ and the type $C_n$ Gale order on admissible subsets of $[\pm n]$.
\end{proposition}

\begin{proof}
    For any $n$-symmetric lattice path $p=\alpha_1\ldots \alpha_{2n}$, we have that $\labL(p)$ is an $n$-admissible subset of $[\pm n]$ since the symmetry condition implies that $\alpha_{i} \not = \alpha_{n+i}$ and hence exactly one of $i$ and $-i$ is in $\labL(p)$ for all $i \in [n]$. Further, every $n$-admissible subset appears in this way since both sets have size $2^n$. In particular, $p$ is determined uniquely by $\alpha_{1} \ldots \alpha_n$, which can be freely chosen.

    Suppose that $S$ and $T$ are two admissible subsets of $[\pm n]$ of cardinality $n$ such that $S \leq T$ with corresponding symmetric lattice paths $p$ and $q$. To show that this function is a poset isomorphism, we need to show that $p$ is weakly below $q$. 

    Let $p = \alpha_1 \cdots \alpha_{2n}$ and $q = \beta_1 \cdots \beta_{2n}$. Then, we need to show that the number of eastward steps in $\alpha_k \cdots \alpha_{2n}$ is smaller than the number of eastward steps in $\beta_k \cdots \beta_{2n}$. By symmetry, it suffices to check this for $k \geq n+1$. Then, the number of eastward steps in $p$ and $q$ starting at $k$ is the same as the integers $|A \cap \{k, \ldots, n\}|$ and $|B \cap \{k,\ldots, n\}|$, where $A = S \cap [n]$ and $B = T \cap [n]$. By Lemma \ref{lem: Gale order and truncated sets}, this is equivalent to $A \leq B$ in the type $C_n$ Gale order, which is in turn equivalent to $S \leq T$.
\end{proof}

\begin{definition}
    Let $p$ and $q$ be two $n$-symmetric lattice paths. The \textbf{lattice path delta matroid} $\Delta[p,q]$ is the lattice path matroid whose feasible sets are $\labL(r) \cap [n]$ for every symmetric lattice path $r$ in the interval $[p,q]$. Likewise, the corresponding Lagrangian matroid $\mathcal{L}[p,q]$ is the Lagrangian matroid whose bases are $\labL(r)$ for $p \leq r \leq q$.
\end{definition}

\begin{figure}[h!]
\centering
\scalebox{0.9}{
\begin{tikzpicture}[inner sep=0in,outer sep=0in]
\node (n) {\begin{varwidth}{5cm}{
\ydiagram{5,5,5,5,5}}\end{varwidth}};
\draw[very thick,blue] (n.south west)--([xshift=0.8cm]n.south west)--([xshift=\boxs,yshift=0.5*\boxs]n.west)--([xshift=-3*\boxs - 0.5,yshift=0.5*\boxs]n.east)--([xshift=-3*\boxs-0.5,yshift = 1.5*\boxs + 0.5]n.east)--([yshift = -1*\boxs]n.north east) -- (n.north east);
\draw[very thick, black] (n.north west) -- (n.south east);
\end{tikzpicture}
}
\caption{Symmetric lattice path and delta lattice path of the admissible subset $\overline{5}\overline{1} 2 3 4$.}
\end{figure}

Given an $(a,b)$-lattice path $p$, we let $\lambda(p)$ be the Young diagram below the path $p$ contained in the box $[0,a] \times [0,b]$. If $p$ is a symmetric lattice path, then $\lambda(p)$ is a Young diagram which is symmetric along the line $y = n - x$. In other words, it is a self-conjugate partition. We will often draw the lattice path matroid $\Delta[p,q]$ by its symmetric skew-diagram $\lambda(q)/ \lambda(p)$.

It will sometimes be convenient to study lattice path delta matroids on ground sets other than $[n]$. For this, let $E$ be a finite set with a given total order $\leq$. Let $E^{\pm}$ denote the set $E \sqcup \overline{E}$ where $\overline{E} = \{ \overline{e} \; \lvert \; e \in E\}$. We endow $E^{\pm}$ with the total order determined by $\overline{e} \leq e'$ for all $e,e' \in E$ and $\overline{e} \leq \overline{e}'$ if and only if $e' \leq e$ for all $e,e' \in E$.

Suppose $E^{\pm} = \{e_{i_1} < \cdots < e_{i_{2n}}\}$. Given a symmetric lattice path $p = \alpha_{1} \cdots \alpha_{2n}$ in $[0,n]^2$, we can obtain an admissible subset $S$ of $E^{\pm}$ by labeling the steps by $E^{\pm}$ using the induced order and then recording the eastward steps. Then, a pair of symmetric lattice paths gives a Lagrangian matroid on ground set $E^{\pm}$ whose bases are the labels of all symmetric paths $p \leq r \leq q$. The corresponding delta matroid is the set of subsets of $E$ obtained by $B \mapsto B \cap E$ for $B \in \mathcal{B}(\mathcal{L})$.

\subsection{Type C Catalan Matroid}

One of the first studied lattice path matroids was the Catalan matroid as defined by Ardila in \cite{Ardila2003}. Recall that an $n$-\textbf{Dyck path} is a $(n,n)$-lattice path such that the path is weakly below than the line $y = x$. It is well-known that the number of $n$-Dyck paths is the famous Catalan number $C_n = \frac{1}{n+1} \binom{2n}{n}$. The \textbf{Catalan matroid} is the lattice path matroid on ground set $[2n]$ of rank $n$ given by $M[p,q]$, where $p$ is the bottom path $E^nN^n$ and $q$ is the staircase path $NENE \ldots NE$. The bases are then exactly the Dyck paths. Ardila showed that this matroid satisfies various nice properties and has many numerical invariants related to Catalan combinatorics. We now discuss a type $C_n$ version of this statement.

The \textbf{type $C_n$ Catalan numbers} are the central binomial coefficients $\binom{2n}{n}$. These enumerate various different invariants of the Coxeter group $C_n$ including the number of ideals in the root poset of type $C_n$ and non-crossing partitions of type $C_n$, see \cite{Reiner1997}. Stump showed\footnote{His construction only considers the truncated part of the symmetric path contained below the main diagonal but this completely determines the symmetric path.} that this number also counts the number of Dyck paths from $(0,0)$ to $(2n,2n)$ that are symmetric along the line $y = 2n - x$ \cite{StumpDyck}. We refer to such a path as a \textbf{symmetric Dyck path}.

\begin{definition}\label{def: Type B Catalan Matroid}
    Let $p$ be the bottom bounding path of the box $[0,2n]^2$ and $q$ be the staircase path $(NE)^{2n}$. The \textbf{type $C_n$ Catalan matroid} is the lattice path delta matroid $\Delta[p,q]$.
\end{definition}

Since the feasible sets of this matroid are in bijection with symmetric Dyck paths of length $4n$, we have that the number of feasible sets is the $C_n$-Catalan number.

\subsection{Representability and the Lagrangian Grassmannian}

Let $V$ be a vector space over a field $\FF$ with basis $\{e_1,e_2, \ldots, e_n, e_{\overline{1}}, \ldots, e_{\overline{n}}\}$ and with an anti-symmetric bilinear form $\langle -,- \rangle$.
A subspace $U \subseteq V$ is an \textbf{Lagrangian subspace} of $V$ if $\langle u,v \rangle = 0$ for all $u,v \in U$ and $\dim U = n$. The \textbf{Lagrangian Grassmannian} $\operatorname{LGr}_{n,2n}(\FF)$ is the space defined as the set of all Lagrangian subspaces of $V$.

We can represent each $U \in \operatorname{LGr}_{n,2n}(\FF)$ as an $n \times 2n$ matrix $M$ with rowspace equal to $U$. Let $A$ and $B$ be the submatrices of $M$ obtained by restricting $M$ to the columns $1, \ldots, n$ and $n+1, \ldots , 2n$, respectively. Then a matrix $M$ has a Lagrangian subspace as its rowspace if and only if it is full-rank and $AB^T$ is symmetric.

Label the columns of $M$ in order by $1,2,\ldots,n,\overline{n},\ldots,\overline{1}$. 

\begin{definition}
    For any $U \in \operatorname{LGr}_{n,2n}(\FF)$ with a representing matrix $M$, we say that the Lagrangian matroid $\mathcal{L}(U)$ of $U$ is the Lagrangian matroid whose bases are the admissible subsets $S \subseteq [\pm n]$ of size $n$ such that the minor $\det(M_S)$ of $M$ obtained by restricting to the columns labeled by $S$ is non-zero.

    We say that a Lagrangian matroid $\mathcal{L}$ is \textbf{representable} over $\FF$ if there exists an $U \in \operatorname{LGr}_{n,2n}(\FF)$ such that $\mathcal{L} = \mathcal{L}(U)$.
\end{definition}

Given an admissible subset $A$, let \textbf{the Schubert variety} $\Omega_A$ associated to $A$ be the closure of the set of $U \in \operatorname{LGr}_{n,2n}$ such that the smallest non-vanishing minor $M_S$ of an admissible subset $S$ is $A$ (under the type $C_n$ Gale order). The \textbf{opposite Schubert variety} $\Omega^A$ is likewise defined by insisting that $M_A$ is the largest non-vanishing minor. Given $A \leq B$ admissible subsets, the \textbf{Richardson variety} $X_A^B$ is the intersection $\Omega_A \cap \Omega^B$.

It is a well-known fact that a generic point in the Richardson variety $U \in X_A^B$ has $\mathcal{L}(U) = \mathcal{L}[A,B]$. See \cite{Tsukerman2015} for a general description for flag varieties $G/P$. Therefore, all of our lattice path matroids are representable. Under the identification $\operatorname{LGr}_{n,2n}(\RR) \cong U(n)/O(n)$, $P(\mathcal{L}[A,B])$ is the image of the moment map of $X_A^B$ corresponding to the action of the maximal torus of $U(n)$.

\section{First Properties and Operations}\label{Sec: First Properties}

\subsection{Operations on Lattice Path Delta Matroids}

Matroids have many important operations which have been shown to preserve the class of lattice path matroids \cite{Bonin2003}. We now prove similar results for operations on delta matroids. 

\begin{definition}
Let $\Delta$ be a delta matroid on $[n]$.
\begin{itemize}

    \item The \textbf{deletion} of $\Delta$ by the subset $T \subseteq [n]$ is the delta matroid $\Delta \backslash T$ given by 
     \[\mathcal{F}(\Delta\backslash T) = \{B \in \mathcal{F}(\Delta) \; \lvert \; B \cap T = \emptyset\} \]
     
     \item The \textbf{contraction} of $\Delta$ by the subset $T \subseteq [n]$ is the delta matroid $\Delta /T$ on $[n] \backslash T$ with feasible sets
     \[ \mathcal{F}(\Delta/T)=\{ B \backslash T \; \lvert \; B \in \mathcal{F}(\Delta), B \supseteq T \}. \]

     \item The \textbf{direct sum} of two delta matroids $\Delta_1$ and $\Delta_2$ on ground sets $E_1$ and $E_2$, respectively, is the delta matroid with feasible sets
      \[\mathcal{F}(\Delta_1 \oplus \Delta_2) = \{B_1 \cup B_2 \; \lvert \; B_1 \in \mathcal{F}(\Delta_1), B_2 \in \mathcal{F}(\Delta_2)\}. \]

    \item The \textbf{dual} of $\Delta$ is the delta matroid $\Delta^*$ with feasible sets
     \[ \mathcal{F}(\Delta^*) = \{ [n] \backslash B \; \lvert \; B \in \mathcal{F}(\Delta)\}.\]

\end{itemize}
\end{definition}

We have the usual duality relation between deletion and contraction $(\Delta \backslash \ell)^* = \Delta^* /\ell$.

A \textbf{loop} of a delta matroid is an element $i \in [n]$ that is not in any feasible set, and a \textbf{coloop} is an element $i \in [n]$ that is in every feasible set.

\begin{proposition}
 Let $\Delta[S,T]$ be a lattice path delta matroid. Then, the dual $\Delta[S,T]^*$ is the lattice path delta matroid $\Delta[[n] \backslash T, [n] \backslash S]$.    
\end{proposition}

\begin{proof}
    It is clear that the map $A \mapsto [n] \backslash A$ is an order-reversing bijection from $[n]$ to $[n]$. Therefore, if $S \leq A \leq T$, then $[n] \backslash T \leq [n] \backslash A \leq [n] \backslash S$.
\end{proof}

In terms of symmetric lattice paths, duality corresponds to reflecting the path along the main anti-diagonal $y = x$. We now turn to deletion and contraction.

\begin{lemma}
\label{lem: alternate coloop def}
    Let $\Delta[S,T]$ be a lattice path delta matroid corresponding to the subsets $S,T \subseteq [n]$ and $i \in [n]$.
    Then $\ell$ is a coloop of $\Delta[S,T]$ if and only if $\ell \in S \cap T$ and $|S_{\geq \ell+1}|=|T_{\geq \ell+1}|$, or equivalently,  $|S_{\geq \ell}| = |T_{\geq \ell}| = |S_{\geq \ell + 1}| + 1 = |S_{\geq \ell + 1}| + 1$.
\end{lemma}
\begin{proof}
    Let $\ell\in S \cap T$ and $|S_{\geq {\ell+1}}|=|T_{\geq {\ell+1}}|=a$. Then $|S_{\geq \ell}|=|T_{\geq \ell}|=a+1$.
    By Lemma \ref{lem: Gale order and truncated sets}, for any $R \in [S,T]$, since $S\leq R\leq T$, we must have $|R_{\geq \ell}|=a+1$ and $|R_{\geq \ell+1}|=a$, forcing $\ell \in R$.  Therefore, $\ell$ is a coloop.

    We prove the other direction through its contrapositive. Consider a maximal chain
     \[ S = S^0 \lessdot S^1 \lessdot \cdots \lessdot S^m = T.\]
    By Lemma \ref{lem: rank calculation of Gale order on subsets}, we know that $A \lessdot B$ if and only if there is some $i$ such that $i \in A$, $i+1 \not \in A$, and $B = A \backslash \{i\} \cup \{i+1\}$. Hence at each step, there is an element $a$ such that $S^{i+1} = S^{i} \backslash \{a\} \cup \{a +1\}$.

    Clearly if $\ell \not \in S \cap T$, then $\ell$ is not a coloop. Suppose we have $|S_{\geq \ell +1}| < |T_{\geq \ell}|$. Since $|S^0_{\geq \ell +1}| < |S^m_{\geq \ell +1}|$, we have that there must exist some $i$ such that
     \[|S^{i}_{\geq \ell +1}| < |S^{i+1}_{\geq \ell +1}|. \]
    At this step, $S^{i+1} = S^{i} \cup \{\ell+1\} \backslash \{\ell\}$. This implies that $\ell \not \in S^{i+1}$ and hence $\ell$ is not a coloop.

\end{proof}

\begin{proposition} \label{prop: deletion }
    Let $\Delta[S,T]$ be a lattice path delta matroid corresponding to the subsets $S,T \subseteq [n]$. Let $\ell \in [n]$ be an element that is not a coloop of $\Delta[S,T]$. 

    Let $A = S$ if $\ell \not \in S$ and
     \[A = S \backslash \ell \cup \min(i \not \in S \;\lvert \; \ell < i) \]
    otherwise. Likewise, let $B = T$ if $\ell \not \in T$ and
     \[B= T \backslash \ell \cup \max(i \not \in T \; \lvert \; i < \ell) \]
    otherwise.
    
    Then, $\Delta[S,T] \backslash \ell$ is the lattice path delta matroid $\Delta[A,B]$ on ground set $[n] \backslash \ell$.
\end{proposition}

\begin{proof} 
    Note first the inequality $|S_{\geq i}| \leq |S_{\geq {i+j}}| + j$ holds for all $i \in [n]$ and $j \in [n-i]$. 
    We claim that if $\ell$ is not a coloop, then we have the following strict inequalities:
    \[
    \begin{cases}
        |S_{\geq \ell}| < |T_{\geq \ell}| & \text{if} \; \ell \in T \\
        |S_{\geq {\ell+1}}| < |T_{\geq {\ell + 1}}| & \text{if} \; \ell \in S.
    \end{cases}
    \]

    To show the first inequality, assume that $|S_{\geq \ell}| = |T_{\geq \ell}|$ and $\ell \in T$. 
    Then $|S_{\geq \ell +1}| \geq |S_{\geq \ell}| - 1 = |T_{\geq \ell}| - 1= |T_{\geq \ell + 1}|$. But Lemma \ref{lem: Gale order and truncated sets} ensures $|S_{\geq \ell + 1}| \leq |T_{\geq \ell + 1}|$, so $|S_{\geq \ell + 1}|=|T_{\geq \ell+1}|=|S_{\geq \ell}|-1$ and $\ell \in S \cap T$.
    By Lemma \ref{lem: alternate coloop def}, $\ell$ is a coloop. 
    
    Similarly, for the second inequality, assume that $|S_{\geq \ell + 1}| = |T_{\geq \ell + 1}|$ and $\ell \in S$. We have $|S_{\geq \ell}| = |S_{\geq \ell + 1}| + 1 = |T_{\geq \ell + 1}| + 1 \geq |T_{\geq \ell}|$. But $|S_{\geq \ell}| \leq |T_{\geq \ell}|$, so it again follows from Lemma \ref{lem: alternate coloop def} that $\ell$ is a coloop.
    
    If $\ell \in S$, then let $\alpha = \min \{ a \geq 0 \; \lvert \; \ell + a \in S, \; \ell + a + 1 \not\in S\}$, so that $A = S \backslash \ell \cup (\ell + \alpha + 1)$.
    For any $i \in [n]$, we have 
    \[ |A_{\geq i}| = 
    \begin{cases}
        |S_{\geq i}| + 1 & \text{if } \ell+1\leq i\leq \ell+\alpha+1 \\
        |S_{\geq i}| & \text{otherwise}.
    \end{cases} \] 
    
    Similarly, if $\ell \in T$, then let $\beta = \min \{ b \geq 0 \; \lvert \; \ell - b \in T, \; \ell - b - 1 \not\in T\}$, so that $B = T \backslash \ell \cup (\ell - \beta - 1)$.
    For any $i \in [n]$, we have
    \[ |B_{\geq i}| = 
    \begin{cases}
        |T_{\geq i}| - 1 & \text{if } \ell - \beta \leq i \leq \ell \\
        |T_{\geq i}| & \text{otherwise}.
    \end{cases} \]
    
    Then Lemma \ref{lem: Gale order and truncated sets} implies that $A \geq S$ and $B \leq T$.
    We now proceed to show that $|A_{\geq i}| \leq |B_{\geq i}|$, which yields $A\leq B$ via Lemma \ref{lem: Gale order and truncated sets}.

    It suffices to show $A \leq B$ for $i \in [\ell + 1, \ell + \alpha + 1]$ and for $i \in [\ell - \beta, \ell]$ when $\alpha$ and $\beta$ are defined. When $\ell \in S$ and $i \in [\ell + 1, \ell + \alpha + 1]$, we have $|S_{\geq \ell + 1}| < |T_{\geq \ell + 1}|$, and $\{i, i + 1, \dots, \ell + \alpha\} \subseteq S$ by the definition of $\alpha$. 
    Then $|T_{\geq \ell+1}|\geq |S_{\geq \ell+1}|+1=|S_{\geq i}|+i-\ell=|S_{\geq i-1}|+i-\ell-1$ since $i-1 \in S$.
    Thus, $|S_{\geq {i-1}}| \leq |T_{\geq \ell+1}|-(i-\ell-1) \leq |T_{\geq i}|$.
    We then have $|S_{\geq i-1}|=|S_{\geq i}|+1=|A_{\geq i}| \leq |T_{\geq i}|=|B_{\geq i}|$.
    
    Next, when $\ell \in T$ and $i \in [\ell - \beta, \ell]$, we have $|S_{\geq \ell}| < |T_{\geq \ell}|$, and $\{i,i+1,\dots,\ell-1\} \subseteq B$ by the definition of $\beta$, so $|B_{\geq i}|=|B_{\geq \ell}|+\ell - i$.
    Since $|B_{\geq \ell}| = |T_{\geq \ell}| - 1$, we then have $|B_{\geq \ell}| \geq |S_{\geq \ell}|$. 
    Thus, $|B_{\geq i}| \geq |S_{\geq \ell}|+\ell - i \geq |S_{\geq i}| = |A_{\geq i}|$.
    
    Therefore, we conclude that $A \leq B$.
\end{proof}

We also have the dual statement.

\begin{proposition}
    Let $\Delta[S,T]$ be a lattice path delta matroid. Let $\ell \in [n]$ be an element that is not a loop. 

    Let $A = S\backslash \ell$ if $\ell \in S$ and
     \[A = S \backslash \max(i \in S \;\lvert \; i < \ell) \]
    otherwise. Likewise, let $B = T\backslash \ell$ if $\ell \in T$ and
     \[B= T \backslash \min(i  \in T \; \lvert \; \ell < i) \]
    otherwise.

    Then, $\Delta[S,T]/{\ell}$ is the lattice path delta matroid $\Delta[A,B]$ on ground set $[n]\backslash \ell$.
\end{proposition}

\begin{proof}
    This follows from the previous proposition since duality exchanges deletion and contraction by $(\Delta \backslash \ell)^* = \Delta^* / \ell$ and from the description of duality in terms of set complements.
\end{proof}

We have to be careful with direct sums since lattice path delta matroids depend heavily on the underlying total order of the ground set.

\begin{proposition}
    Let $\Delta = \Delta[A,B]$ and $\Delta' = \Delta[S,T]$ be lattice path delta matroids on ground sets $(E_1,\leq_1)$ and $(E_2,\leq_2)$. Then, $\Delta \oplus \Delta'$ is the lattice path delta matroid $\Delta[A \sqcup S, B \sqcup T]$ on ground set $E_1 \sqcup E_2$ with order $\leq$ obtained by extending the orders for $E_1$ and $E_2$ by $x \leq y$ for all $x \in E_1$ and $y \in E_2$.
\end{proposition}

\begin{proposition}
    The set of lattice path delta matroids is closed under deletion, contraction, duality, and direct sums.
\end{proposition}

\begin{proof}
    This is a summary of the results of this section.
\end{proof}

\begin{remark}
    Another commonly studied operation on delta matroids is the projection map $\operatorname{proj}_i$ which for a delta matroid $\Delta$ on ground set $[n]$ gives the delta matroid with feasible sets
    \[\mathcal{F}(\operatorname{proj}_i(\Delta)) = \{B \backslash \{i\} \; \lvert \; B \in \mathcal{F}(\Delta)\} \]
    on ground set $[n] \backslash i$.

    Lattice path delta matroids are not closed under this operation. To see this, consider the lattice path delta matroid $\Delta[34,235]$. The feasible sets are
     \[34,134,35,135,234,235,45,145,245. \]
    The projection $\operatorname{proj}_4(\Delta[34,235])$ has feasible sets
     \[3,5,13,23,15,25,35,135,235. \]
    Note that $3 \leq 125 \leq 235$ but $125$ is not a feasible set in this projection.

\end{remark}

\subsection{Enveloping Matroids}

\begin{definition}
    Let $\mathcal{L}$ be a Lagrangian matroid on ground set $[\pm n]$. Then a matroid $M$ on ground set $[\pm n]$ is an \textbf{enveloping matroid} of $\mathcal{L}$ if the image of $P(M)$ under the map $\pi: \RR^{2n} \to \RR^n$ given by
    \[ \pi(x_1,\ldots,x_n,x_{\overline{1}},\ldots,x_{\overline{n}}) = \frac{1}{2} \left ((x_1 - x_{\overline{1}}, \ldots, x_n -x_{\overline{n}}) + (1,1,\ldots,1)\right )\]
    is $P(\mathcal{L})$.
\end{definition}

This concept has recently seen use in the extension of the Hodge theory of matroids to type $C_n$ Coxeter matroids \cite{EFLS22}.

\begin{definition}
    Let $\mathcal{L} = \mathcal{L}[p,q]$ be a symmetric lattice path matroid on ground set $[\pm n]$. The \textbf{standard matroid envelope} of $\mathcal{L}$ is the lattice path matroid $M[p,q]$.
\end{definition}
We are viewing $M[p,q]$ as a lattice path matroid on the ground set $[\pm n]$ with the linear order $-n, -(n-1), \ldots, -1, 1, 2, \ldots, n$. The standard matroid envelope is obtained by adding all non-symmetric lattice paths as bases. We now check that this is indeed an enveloping matroid.
\begin{proposition}
    The standard matroid envelope $M[p,q]$ of $\mathcal{L}[p,q]$ is an enveloping matroid of $\mathcal{L}[p,q]$.
\end{proposition}

\begin{proof}

    If $S \in \mathcal{B}(M[p,q])$ is an admissible subset of size $n$, then $\pi(e_S) = e_{S \cap [n]}$. Since $S$ is an admissible subset of $n$ and its corresponding path is in the interval $[p,q]$, we have that $S$ is a feasible set of $\mathcal{L}[p,q]$ and hence $\pi(e_S)$ is a vertex of $P(\mathcal{L}[p,q])$. 

    Now consider $S \in \mathcal{B}(M[p,q])$ that is not an admissible subset of $[\pm n]$. Notice that if $i$ and $\overline{i}$ are both in $S$, or if neither is in $S$, then $\pi(e_S)_{i} = \frac{1}{2}$. We will modify $S$ in two ways to obtain admissible sets. Define $S^{t}$ by
 \begin{itemize}
     \item For all $i \in [\pm n]$, $i \in S^{t}$ if $i \in S$ and $|\{i, -i\} \cap S| = 1$.
     \item For all $i \in [n]$, $-i \in S^{t}$ if $|\{i, -i\} \cap S| = 0$.
     \item for all $i \in [n]$, $i \in S^{t}$ if $|\{i,-i\} \cap S| = 2$.
 \end{itemize}
    Likewise, define $S^{b}$ by
 \begin{itemize}
     \item For all $i \in [\pm n]$, $i \in S^{b}$ if $i \in S$ and $|\{i, -i\} \cap S| = 1$.
     \item For all $i \in [n]$, $i \in S^{b}$ if $|\{i, -i\} \cap S| = 0$.
     \item for all $i \in [n]$, $-i \in S^{b}$ if $|\{i,-i\} \cap S| = 2$.
 \end{itemize}
By construction, both $S^{t}$ and $S^{b}$ are admissible subsets with the property that $\pi(e_{S^{t}})_i = \pi(e_{S^{b}})_i = \pi(e_S)_i$ whenever $|S \cap \{i, -i\}| = 1$. In the other cases, we have that one of $\pi(e_{S^{t}})_i$ and $\pi(e_{S^{b}})_i$ is $1$ and the other is $0$. Therefore, $\pi(e_S)$ lies in the line segment between $\pi(e_{S^t})$ and $\pi(e_{S^b})$. 

Geometrically, the symmetric lattice path $s^t$ associated to $S^t$ is obtained by taking the segment of the lattice path $s$ associated to $S$ , intersecting it with the half-space above the line $y = n-x$, and then reflecting along the line. Likewise, $S^b$ corresponds to the path $s^b$ obtained by intersecting with the bottom half and then reflecting. Since $s$ is a lattice path contained between the two symmetric lattice paths $p$ and $q$ we have that both $s^t$ and $s^b$ are symmetric lattice paths contained between $p$ and $q$.

This implies that $\pi(e_{S^t})$ and $\pi(e_{S^b})$ are vertices of $P(\mathcal{L}[p,q])$. By convexity, we have that $\pi(e_S) \in P(\mathcal{L}[p,q])$.
\end{proof}

\begin{example}
    The type $A$ Catalan matroid on ground set $[4n]$ is isomorphic to the standard enveloping matroid of the type $C$ Catalan matroid on ground set $[2n]$.
\end{example}

\subsection{Homogeneous Components}

\begin{definition}
    Let $\Delta$ be a delta matroid on ground set $[n]$. The $k$th \textbf{homogeneous component} $\Delta^k$ of $\Delta$ is the matroid with bases
     \[ \mathcal{B}(\Delta^k) = \{ B \in \mathcal{F}(\Delta) \; \lvert \; |B| = k\}.\]
\end{definition}

We can obtain the $k$th homogeneous component polytopally by intersecting with a specific hyperplane.
\begin{proposition}\label{prop: truncation hyperplane description}
    Let $\Delta$ be a delta matroid and $H_k$ be the hyperplane given by
     \[H_k = \{ x \in \RR^n \; \lvert \; x_1 + x_2 + \cdots +x_n = k\}.\]
    Then, $H_k \cap P(\Delta) = P(\Delta^k)$.
\end{proposition}

\begin{proof}

The definition implies that every vertex of $P(\Delta^k)$ is a vertex of $P(\Delta)$ contained in the hyperplane $H_k$. All that remains to show is that the intersection of $H_k$ with $P(\Delta)$ does not create any new vertices. 

Let $v$ be a vertex of $H_k \cap P(\Delta)$. Then, it must be contained in an edge of $P(\Delta)$. Let $v_1$ and $v_2$ be the vertices of $P(\Delta)$ containing the point $v$. We will show that if $v \in H_k$ then so is $v_1$ and $v_2$ and so $v$ must be $v_1$ or $v_2$.

By Theorem \ref{thm: type B GGMS}, we have that all the edge directions of $P(\Delta)$ are $e_i - e_j$ or $e_i$. If the edge between $v_1$ and $v_2$ has direction $e_i - e_j$, then $v_1 - v_2 = e_i - e_j$ and so every point in the edge has the same coordinate sum. Thus, if $v \in H_k$, then so is $v_1$ and $v_2$. Now suppose that the edge has direction $e_i$ for some $i \in [n]$. Then, the only points in the edge between $v_1$ and $v_2$ that have an integer coordinate sum are the two points $v_1$ and $v_2$. Therefore, the only points in this edge that could be in $H_k$ are the two vertices $v_1$ and $v_2$ of $P(\Delta)$. This shows that the vertices of $H_k \cap P(\Delta)$ are exactly the vertices of $P(\Delta^k)$.
\end{proof}

The homogeneous component of a lattice path matroid has a nice lattice path interpretation.

\begin{proposition}\label{prop: truncation in terms of lattice paths}
    Let $\Delta = \Delta[p,q]$ be an lattice path delta matroid on ground set $[\pm n]$. Then, $\Delta^k$ is the matroid whose bases are in bijection with all symmetric lattice paths that pass through the point $(n-k,k)$.
\end{proposition}

\begin{proof}
    The steps of the paths of $p$ and $q$ that have labels in $[n]$ are exactly those which are above the main diagonal $y = n - x$. Therefore, to obtain the bases of $\Delta^k$ we only need to record the eastward steps of the path above the main diagonal. For this path to end at $(n,n)$ and have exactly $k$ eastward steps above the main diagonal, it must intersect the main diagonal at $(n-k,k)$. Hence, the symmetric lattice paths which give bases for $\Delta^k$ are exactly given the paths in $[p,q]$ that pass through $(n-k,k)$. 
\end{proof}

\begin{corollary} \label{cor: Truncations are LPMs}
    Let $\Delta = \Delta[p,q]$ be a lattice path delta matroid. Then, $\Delta^k$ is a lattice path matroid, when it is non-empty.
\end{corollary}

\begin{proof}
    To choose a symmetric lattice path in $[p,q]$ it suffices to choose a lattice path starting at some $(n-a,a)$ that is contained between the truncations of $p$ and $q$ above the line $y = n - x$ since we can reflect this path to obtain a symmetric lattice path. Therefore, when the truncation is non-empty, the bases of $\Delta[p,q]$ are the lattice paths in the Young diagram obtained by intersecting the skew-diagram $\lambda(p,q)$ with the box with southwest corner $(n-k,k)$ and northeast corner $(n,n)$. The bounding paths of this skew-diagram give the bounding paths of the lattice path matroid $\Delta^k$.
\end{proof}

\begin{example}
    Consider the lattice path delta matroid $\Delta[135,2456]$ on ground set $[6]$. Then, the shaded region below gives the skew-diagram for the lattice path matroids obtained as truncations at $k = 3$ and $k = 4$. These are the matroids $M[135,456]$ and $M[1235,2456]$.

\begin{figure}[H]
\centering
\scalebox{0.9}{
\begin{tikzpicture}[inner sep=0in,outer sep=0in]
\node (n) {\begin{varwidth}{5cm}{
\ydiagram[*(green)]
  {3+3,3+2,3+1}
*[*(white)]{3+3,2+3,1+3,3,2,1}}\end{varwidth}};
\draw[very thick, black] (n.north west) -- (n.south east);
\end{tikzpicture}
\begin{tikzpicture}[inner sep=0in,outer sep=0in]
\node (n) {\begin{varwidth}{5cm}{
\ydiagram[*(green)]
  {3+3,2+3}
*[*(white)]{3+3,2+3,1+3,3,2,1}}\end{varwidth}};
\draw[very thick, black] (n.north west) -- (n.south east);
\end{tikzpicture}
}
\caption{The skew-diagrams of the lattice path matroids corresponding to the truncations at $k = 3$ and $k =4$.}
\label{fig: truncation of LPDM}
\end{figure}
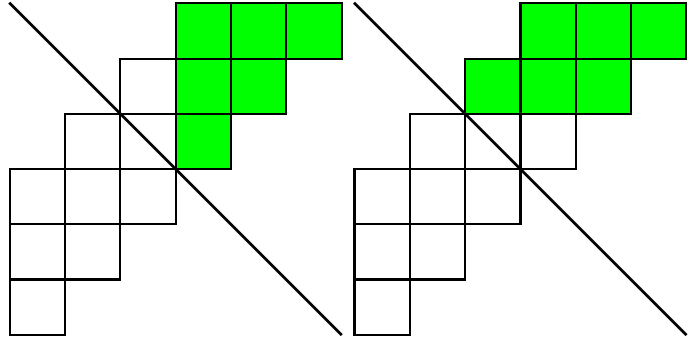

\end{example}

\section{Delta Matroid Polytopes}\label{Sec: Polytopes}

\subsection{Hyperplane Description}

We now give hyperplane descriptions of lattice path delta matroid polytopes. Given a subset $S$ of $E$, we use the notation $S_{\geq i} = \{s \in S \; \lvert \; s \geq i\}$. Likewise, for $x \in \RR^E$, let 
\[ x_{\geq i} =\sum_{\substack{j \in E \\ j \geq i}} x_j.\]
For two subsets $A,B$ of a finite set $E$, let $C(A,B)$ denote the polytope in $\RR^E$ determined by the inequalities

     \begin{align} |A_{\geq i}| \leq x_{\geq i} \leq |B_{\geq i}|
     \end{align}
     and
     \begin{align}
         0 \leq x_i \leq 1
     \end{align}
        for all $i \in E$. We will show that $C(S,T) = P(\Delta[S,T])$ for $S,T \subseteq [n]$. We begin with a lemma that describes the intersection of $C(S,T)$ with the hyperplane $\{x \in \RR^n\; \lvert \;x_i = 0\}$.

\begin{lemma}\label{lem: faces of C(S,T)}
    Let $S$ and $T$ be subsets of $[n]$ such that $S \leq T$ in the type $C_n$ Gale order. Let $A$ and $B$ be the subsets of $[n] \backslash i$ defining the lattice path delta matroid $\Delta[A,B] =\Delta[S,T] \backslash \ell$ as in Proposition \ref{prop: deletion }. If $x \in C(S,T)$ satisfies $x_\ell = 0$, then $\pi(x) \in C(A,B)$, where $\pi$ is the projection $\RR^{[n]} \to \RR^{[n] \backslash \ell}$.
\end{lemma}

\begin{proof}
Let $j = \min(i \not \in S \; \lvert \; \ell < i)$ and $m =  \max(i \in S \; \lvert \; i < \ell)$. As in the proof of Proposition \ref{prop: deletion }, if $\ell \not \in S$, then $A = S$. Otherwise,
    \[ |A_{\geq i}| = 
    \begin{cases}
        |S_{\geq i}| + 1 & \text{if $i \in [\ell+1,j]$} \\
        |S_{\geq i}| & \text{otherwise}.
    \end{cases} \] 
    
    Similarly, if $\ell \not \in T$, then $T = B$, Otherwise,
    \[ |B_{\geq i}| = 
    \begin{cases}
        |T_{\geq i}| - 1 & \text{if } i \in [m+1,\ell]\\
        |T_{\geq i}| & \text{otherwise}.
    \end{cases} \]

If $\ell \not \in S$, then the needed lower bounds trivially follow since the bounds are the same. Now, suppose $\ell \in S$. We will show that $x$ satisfies the lower bounds by induction. Using the equations for $C(S,T)$ and $x_{\ell} = 0$, we have the inequality
 \[|S_{\ell}| \leq x_{\ell + 1} + \cdots + x_n. \]
Since $|S_{\ell}| = |S_{\ell+1}| + 1$, we have that the left hand side is bounded by $|A_{\ell+1}|$, which will give the base case. Now suppose that $i \in [\ell+2, j]$, and inequality (1) for $C(A,B)$ holds for $i-1$. Then, we have the inequalities
 \[|S_{\geq i}| \leq x_{i} + \cdots + x_n \quad \quad \text{and} \quad \quad |A_{\geq i-1}| \leq x_{i-1} + \cdots + x_n. \]
Since $0 \leq x_{i-1} \leq 1$, these inequalities together imply
 \[\max(|S_{\geq i}|, |A_{\geq i-1}| - 1) \leq x_{i -1} + \cdots +x_n. \]
Since $i \in [\ell+1,j]$, we know that $i \in S$, so $|A_{\geq i-1}| - 1 = |A_{\geq i}| = |S_{\geq i}| + 1$. This gives the needed inequality.

A similar argument gives the bounds on the other side.
\end{proof}

\begin{theorem}[Theorem \ref{mainthm: hyperplane description}]\label{thm: hyperplane description}
    Let $S$ and $T$ be subsets of $[n]$ with $S \leq T$ in the type $C_n$ Gale order. Then, $P(\Delta[S,T]) = C(S,T)$.
\end{theorem}

\begin{proof}
    By Lemma \ref{lem: Gale order and truncated sets}, we have that a subset $A$ satisfies $S \leq A \leq T$ if and only if $|S_{\geq i}| \leq |A_{\geq i}| \leq |T_{\geq i}|$ for all $i \in [n]$. From this, it follows that $|S_{\geq i}| \leq e_{A_{\geq i}} \leq |T_{\geq i}|$ and so every vertex of $P(\Delta[S,T])$ satisfies these equations. Since these equations are closed under convex combinations, this proves that $P(\Delta[S,T]) \subseteq C(S,T)$.

    For containment in the other direction, suppose that $z$ satisfies the conditions of the proposition. We need to show that $z \in P(\Delta[S,T])$. We prove this by induction on the size of the ground set of $\Delta[S,T]$. The base case is when the ground set has size $1$, in which case the hyperplane description easily holds.

    If $z$ satisfies $z_i = 0$ for some coordinate $i$, then by the inductive hypothesis and Lemma \ref{lem: faces of C(S,T)}, we have that $\pi(z)$ is contained in $P(\Delta[S,T]) \backslash i$, where $\pi: \RR^n \to \RR^{[n] \backslash i}$ is the projection map. Therefore, $z$ is contained in the face of $P(\Delta[S,T])$ intersected with $\{ x \in \RR^n\; \lvert \; x_i = 0\}$. A dual argument can be made for whenever $z_i = 1$.

    Now suppose that $z$ has $z_i \not = 0$ and $z_i \not = 1$ for all $i \in [n]$. Let $m$ be the minimal non-zero value of the coordinates of $x$, and let $M$ be the maximal non-one value of the coordinates of $x$. Since $z_i \neq 0,1$, this means $\Delta[S,T]$ has no loops or coloops.
    
    We proceed by cases:

\vskip 2ex
\noindent
    \textbf{Case 1:} $m \leq 1 - M$
    
    Choose one index $\ell$ such that $z_{\ell} = m$. Since $\Delta[S,T]$ has no loops, we can choose a vertex $v$ of $P(\Delta[S,T])$ such that $v_{\ell} = 1$. Consider the vector $y$ given by
     \[y = \frac{z  - mv}{1-m}. \]
    It is clear that $y$ continues to satisfy equation (1) of the proposition since both $z$ and $v$ satisfy them.
    
    To show that it satisfies equation (2), let $j\not = \ell$ be an index in $[n]$. If $v_j = 1$, then $y_j = \frac{z_j-m}{1-m}$. Since $z_j$ is larger than the minimal value $m$ and is less than $1$, we have that $0 \leq y_j \leq 1$. If $v_j = 0$, then $y_j = \frac{z_j}{1-m}$. Since $z_j$ is bounded above by $M$ and $M \leq 1 - m$, we have that $y_j$ satisfies equation (2).

    By convexity, it suffices to show that $y \in P(\Delta[S,T])$. However, $y_{\ell} = 0$, so we can use the inductive hypothesis as above to show that $y \in P(\Delta[S,T])$.

\vskip 2ex
\noindent
    \textbf{Case 2:} $m \geq 1 - M$

    Choose one index $\ell$ such that $z_{\ell} = M$. Let $v$ be a vertex of $P(\Delta[S,T])$ such that $v_{\ell} = 0$. Then, consider the vector
     \[y = \frac{z - (1-M) v}{M}. \]
    As in the previous case, we can quickly see that $y$ satisfies equation (2) for all $i \in [n]$.

    Let $j\not = \ell \in [n]$. If $v_j = 1$, then $y_j = \frac{z_j - (1- M)}{M}$. Since $z_j$ is larger than $m$, which by the assumption of this case, is larger than $(1-M)$, we have that $0 \leq y_j \leq 1$. If $v_j = 0$, then $y_j = \frac{z_j}{M}$. Hence, $z_j$ is bounded above by $M$ and so $0 \leq y_j \leq 1$.

    By convexity, $z \in P(\Delta[S,T])$ if and only if $y \in P(\Delta[S,T])$. Again notice that $y_{\ell} = 1$ and so the inductive hypothesis implies the result.
\end{proof}

\begin{remark}
    This is a type $C_n$ generalization of Lemma 3.8 in \cite{Bid12}. However, the proof in that paper is incomplete. In particular, Bidkhori only has the type $A_n$ version of case 1 above. The argument presented there fails whenever $m < 1-M$. That proof can be fixed by considering a second case as in our proof. Another approach to this proposition is to study symmetric generalized lattice paths generalizing the proof that appears in \cite{knauer2018lattice}.
\end{remark}

\subsection{Dimension Formula}

From this hyperplane description, we immediately obtain the following description of the dimension of a lattice path delta matroid.

\begin{corollary}\label{cor: dimension formula subsets}
    Let $S,T \subseteq [n]$ with $S\leq T$ in the type $C_n$ Gale order. 
    Then the dimension of the polytope $P(\Delta[S,T])$ is $n-\ell$, where $\ell=\{ i \in [n]:|S_{\geq i}|=|T_{\geq i}| \}$.
\end{corollary}

We can use this to obtain a formula for the dimension of $P(\Delta[p,q])$ in terms of the paths.

\begin{proposition}\label{prop: dimension formula}
    Let $p$ and $q$ be two $n$-symmetric lattice paths. Then the dimension of $P(\Delta[p,q])$ is $n - k + 1$, where $k$ is the number of lattice points in the intersection of $p$ and $q$ weakly above the line $y = n - x$.
\end{proposition}

\begin{proof}
    Let $A = \labL(p)$ and $B = \labL(q)$, and let $S$ and $T$ be the corresponding subsets of $[n]$ in the delta model. By Corollary \ref{cor: dimension formula subsets}, the dimension of the polytope $P(\Delta[p,q]) = P(\Delta[S,T])$ is $n - \ell$, where $\ell$ is the number of $i \in [n]$ such that $|S_{\geq i}| = |T_{\geq i}|$.

    Let $p = \alpha_1 \cdots \alpha_{2n}$ and $q = \beta_1 \cdots \beta_{2n}$. Translating through the various bijections, we see that $|S_{ \geq i}| = | T_{\geq i}|$ occurs exactly when the number of eastward steps in $\alpha_{2n - i} \cdots \alpha_{2n}$ is equal to the number of eastward steps in $\beta_{2n - i} \cdots \beta_{2n}$. Say that the number of eastward steps in this word is $a$ and the number of northward steps is $b$. Then, the two paths intersect at the lattice point $(n-a, n-b)$ which is weakly above the line $y = n - x$ since we are only considering the suffixes of the words corresponding to $p$ and $q$ of length less than $n$.
\end{proof}

We say that an interval $[p,q]$ is \textbf{linked} if the polytope $P(\Delta[p,q])$ has dimension $n$. By the previous proposition, this is equivalent to the paths $p$ and $q$ not intersecting outside the start and end points.

\subsection{Intersections and Faces}

\begin{proposition}\label{prop: intersections of LPDM are LPDM}
    Let $\Delta[S,T]$ and $\Delta[S',T']$ be two lattice path delta matroids with feasible polytopes $P = P(\Delta[S,T])$ and $P' = P(\Delta[S',T'])$. The intersection $P \cap P'$ is either empty or the feasible polytope of a lattice path delta matroid.
\end{proposition}

\begin{proof}
    Let $a_i = |S_{\geq i}|$, $a_i' = |S'_{\geq i}|$, $b_i = |T_{\geq i}|$, and $b_i' = |T'_{\geq i}|$. Define $c_i = \min(a_i,a_i')$ and $d_i = \max(b_i,b_i')$. By Theorem \ref{mainthm: hyperplane description}, we have that $x \in P \cap P'$ if and only if
     \[ c_i \leq x_i + \cdots + x_n \leq d_i\]
    and
     $0 \leq x_i \leq 1$ for all $i \in [n]$. We need to show there are sets $C$ and $D$ with $|C_{\geq i}| = c_i$ and $|D_{\geq i}| = d_i$. By construction, we have that either $b_i = b_{i-1} + 1$ or $b_i = b_{i-1}$ and likewise for $a,a'$, and $b'$. Therefore, we have that $d_i = \max(b_i,b_i')$ is either $d_{i-1}$ or $d_{i-1} + 1$. Indeed if $b_i > b_i'$, then $b_{i-1} \geq b_{i-1}'$ and so $d_i = d_{i-1}$ and likewise for the other cases. A similar argument shows that $c_i$ is either $c_{i-1}$ or $c_{i-1} + 1$. The desired subsets $C$ and $D$ are the subsets where $i \in C$  if and only if $c_i = c_{i-1} + 1$ and likewise for $D$.
     
     If $C$ is not less than $D$ in the type $C_n$ Gale order, then, by Lemma \ref{lem: Gale order and truncated sets}, there exists an index $i \in [n]$ with $|C_{\geq i}| > |D_{\geq i}|$. Then, the equation
      \[ |C_{\geq i}| \leq x_i + \cdots + x_n \leq |D_{\geq i}|\]
    has no solutions. Otherwise, we have that $C \leq D$, and the intersection is the feasible polytope of $\Delta[C,D]$.
\end{proof}

From the proof, it follows that the symmetric skew-diagram $\mu$ corresponding to the lattice path delta matroid whose polytope is the intersection $P \cap P'$ is the intersection of the two symmetric skew-diagrams $\lambda(q)/\lambda(p) \cap \lambda(q')/\lambda(p')$.

We will use the following lemma in a future section.

\begin{lemma}\label{lem: dimension of intersection of snake paths lpdm}
    Let $S$ and $S'$ be two distinct subsets of $[n-1]$. The intersection of the full-dimensional polytopes $P(\Delta[S, S \cup \{n\}])$ and $P(\Delta[S',S'\cup \{n\}])$ is not full-dimensional.
\end{lemma}

\begin{proof}
    Let $a_i = |S_{\geq i}|$ and $a_i' = |S'_{\geq i}|$. Then, it is immediate that the two polytopes are determined by the inequalities
     \[a_i \leq x_i + \cdots + x_n \leq a_i + 1 \quad \quad \text{and} \quad \quad a_i' \leq x_i + \cdots + x_n \leq a_i' + 1, \]
    for all $i \in [n]$. Since $S \not = S'$, there is an index $i$ such that $a_i \not = a_i'$. Then, for the intersection to be non-empty, we must have either
     \[ x_i + \cdots + x_n = a_i \quad \quad \text{or} \quad \quad x_i + \cdots + x_n = a_i'.\]
    Therefore, the intersection of the polytope is contained in a hyperplane and is not full-dimensional.
\end{proof}

\begin{proposition}\label{prop: face are again lattice path delta matroids}
    Let $S \leq T$ be subsets of $[n]$ with corresponding lattice paths $p$ and $q$. Then, any face of $P(\Delta[S,T])$ is the feasible polytope of a delta matroid $\Delta$ such that $\Delta$ is a direct sum of lattice path delta matroids.
\end{proposition}

\begin{proof}
    By Theorem \ref{mainthm: hyperplane description}, we have that $P(\Delta[S,T])$ is determined by the inequalities
     \[ |S_{\geq i}| \leq x_i + \cdots + x_n \leq |T_{\geq i}|\]
    and $0 \leq x_i \leq 1$. By induction, it suffices to prove the statement for the faces of $P(\Delta[S,T])$ obtained by intersecting with a hyperplane of the form
     \[ x_i + \cdots +x_n=|S_{\geq i}|, \quad x_i + \cdots +x_n = |T_{\geq i}|, \quad x_i =0, \quad \text{or} \quad x_i = 1, \]
    for any $i \in [n]$.

    We first show that intersecting with any of these hyperplanes $H$ does not create new vertices. This happens since for any edge $e = \conv(v_1,v_2)$, where $v_1$ and $v_2$ are vertices of $P(\Delta[S,T])$, we have that either $e \subseteq H$ or $H \cap e = v_1$ or $v_2$. By Theorem \ref{thm: type B GGMS}, we know that the edge directions of $P(\Delta[S,T])$ are either $e_i - e_j$ or $e_i$.

    Let $\alpha$ be the linear functional $\alpha(v) = v_i + \cdots + v_{n}$. Let $e$ be an edge in $P(\Delta[S,T])$ with end points $v_1$ and $v_2$. We know that $\alpha(v_1)$ and $\alpha(v_2)$ are integers since the vertices are $0$-$1$ vectors. By the description of the possible edge directions alongside the $0$-$1$ condition, we have that $\alpha(v_1)$ and $\alpha(v_2)$ are either equal or differ by exactly $1$. Therefore, the solutions to $\alpha(v) = k$ for some integer $k$ can either be the entire edge or one of the end points. This implies that intersecting with the first or second hyperplane does not create new vertices. Similar arguments work for the third and fourth hyperplanes. Hence, to compute the face corresponding to the hyperplane, it suffices to compute which vertices satisfy the equalities. 
    
    For the hyperplane $H$ given by the first equality, the vertices of $P(\Delta[S,T])$ in the intersection correspond to the subsets $S \leq A\leq T$ such that $|A_{\geq i}| = |S_{\geq i}|$. In terms of paths, the path $r$ corresponding to $A$ intersects the path $p$ at the $(n + i)$th step. The intersection $H \cap P(\Delta[p,q])$ then consists of all symmetric lattice paths $p \leq r \leq q$ that cross through this point. The set of feasible sets corresponding to these symmetric lattice paths give the delta matroid obtained as the direct sum of two lattice path delta matroids. The second inequality is similar.

    Now consider the hyperplane $H$ determined by $x_i = 1$. The face $P(\Delta[S,T]) \cap H$ is determined by all the vertices whose corresponding subset $R$ contains $i$. Then, the face is the polytope of the direct sum of the lattice path delta matroid $[i,i]$ on ground set $\{i \}$ and the lattice path delta matroid $\Delta[S,T] /i$ on ground set $[n] \backslash i$. While this direct sum is not a lattice path delta matroid (with the standard order on $[n]$), each component is a lattice path delta matroid. The hyperplane corresponding to $x_i = 0$ is similar.

\end{proof}

\begin{example}
    Consider the lattice path delta matroid $\Delta[13,235]$. The feasible sets of this delta matroid are
     \[13,14,23,123,24,15,124,25,125,35,134,135. \]
    Then, the delta matroid whose polytope is the face of $P(\Delta[13,235])$ at $x_3 = 1$ consists of all feasible sets containing $3$ which are
     \[ 13,23,123,134,35,135.\]
    This is not a lattice path delta matroid. However, it is the direct sum of the lattice path delta matroid $[3,3]$ on ground set $\{3\}$ and the lattice path delta matroid $[1,15]$ on ground set $[5] \backslash \{3\}$. The former is a matroid consisting of only a coloop and the latter is the contraction $\Delta[13,235]/3$.
\end{example}

\section{Snake Paths and Triangulations}\label{Sec: Snake Paths and Triangulations}

We now move on to our main calculations of our paper. We will give a combinatorial formula for the volume of the matroid polytopes of lattice path delta matroids which is a type $C_n$ generalization of the work of Bidkhori \cite{Bid12}.

In this section, we will use Stanley's triangulation of the hypercube to compute the volume of a special type of lattice path delta matroid. These are the matroids corresponding to the type $C_n$-analogue of the snake graphs studied in \cite{Bid12}.

\begin{definition}
    A $n$-\textbf{symmetric snake path} is an interval of $n$-symmetric lattice paths $[p,q]$ such that the corresponding skew diagram does not contain a $2 \times 2$ box. We say that a symmetric snake path is \textbf{linked} if the paths only intersect at the points $(0,0)$ and $(n,n)$.
\end{definition}
By Proposition \ref{prop: dimension formula}, we have that a symmetric snake path is linked if and only if the corresponding matroid polytope is full-dimensional.

From this point onward, we will focus primarily on the linked case since the other cases can be obtained through direct sums of delta matroids or by taking homogeneous components and reducing to the type $A_n$ case. 

\subsection{Snake Paths as Toric Intervals}

We begin by describing which intervals $[S,T]$ of subsets of $[n]$ correspond to symmetric snake paths. 

\begin{definition}
    We say that an interval $[S,T]$ in the type $C_n$ Gale order on subsets of $[n]$ is a \textbf{toric interval} if the corresponding Lagrangian Richardson variety $X_{S}^T \subset \operatorname{LGr}(n,2n)(\RR) \cong U(n)/O(n)$ is a toric variety under the action of the maximal torus of $U(n)$.
\end{definition}

Tsukerman and Williams showed in \cite{Tsukerman2015} that this happens exactly when $\ell(T) - \ell(S) = \operatorname{dim}(P(\Delta[S,T]))$, where $\ell$ is the rank function on the type $C_n$ Gale order of $2^{[n]}$. 

We have a useful description of which subsets $S,T \subseteq [n]$ give toric intervals $[S,T]$. Recall that an interval $[S,T]$ is linked when $\dim(P(\Delta[S,T]))= n$.
\begin{proposition}\label{prop: subset description for linked toric intervals}
    Let $S \leq T$ be subsets of $[n]$. The interval $[S,T]$ is a linked toric interval if and only if $S \subseteq [n-1]$ and $T = S \cup \{n\}$. 
\end{proposition}

\begin{proof}
   We first show that the intervals $[S, T]$ are toric intervals when $S \subseteq [n-1]$ and $T = S \cup \{n\}$. By Lemma \ref{lem: rank calculation of Gale order on subsets}, we have that $\ell(S \cup \{n\}) - \ell(S) = n$. By definition, $|S_{\geq i}| < |T_{\geq i}|$ for all $i \in [n]$. By Corollary \ref{cor: dimension formula subsets}, the dimension of $P(\Delta[S,T])$ is $n - \ell = n$, where $\ell$ is the number of $i \in [n]$ such that $|S_{\geq i}| = |T_{\geq i}|$. Therefore, this is a toric interval with a full dimensional matroid polytope.

   Conversely, suppose that $[S,T]$ is a toric interval where $P(\Delta[S,T])$ is full-dimensional. The full-dimensionality condition implies that we must have $|S_{\geq i}| < |T_{\geq i}|$ for all $i \in [n]$ by Corollary \ref{cor: dimension formula subsets}. When $i = n$, we have that $|S_{\geq n}| < |T_{\geq n}|$, which implies that $n \in T$ and $S \subseteq [n-1]$. Write $S=\{s_1 < \cdots < s_k\}$ and $T = \{t_1< \cdots < t_{\ell} < n\}$. It remains to show that $T = S \cup \{n\}$.

   Consider the bipartite graph on $\{s_1,\ldots,s_{k}\} \sqcup\{ t_{1}, \ldots, t_{\ell}\}$, where there is an edge from $s_a$ to $t_b$ if and only if $s_a \leq t_b$. The condition $|S_{\geq i}| < |T_{\geq i}|$ for all $i \in [n]$ implies, through Hall's marriage theorem, that there is a perfect matching from $\{s_1, \ldots, s_k\}$ to $\{t_1, \ldots, t_{\ell}\}$. This gives an injective function $\alpha: [k] \to [\ell]$ such that $s_i \leq t_{\alpha(i)}$. Therefore, we have
    \[s_1 + \cdots + s_k \leq t_{\alpha(1)} + \cdots + t_{\alpha(k)}. \]

   Since $[S,T]$ is a toric interval, we have that $\ell(T) - \ell(S) = n$. By Lemma \ref{lem: rank calculation of Gale order on subsets}, we have that $\ell(T) = t_1 + \dots + t_\ell + n$ and $\ell(S) = s_1 + \cdots + s_k$, which implies
   
    \[s_1 + \cdots + s_k = t_1 + \cdots + t_{\ell}. \]

    Since, each $s_i$ and $t_i$ is a non-negative integer, we thus have  $t_{\alpha(i)} = s_i$ and if $\ell = k$. This shows that $T = S \cup \{n \}$.
\end{proof}

\begin{corollary}\label{cor: hyperplane description of toric intervals}
    Let $[S,T]$ be a linked toric interval for subsets in $[n]$. Then, $x \in P(\Delta([S,T])$ if and only if $0 \leq x_i\leq 1$ for all $i \in [n]$ and
     \[|S_{\geq i}| \leq x_i + \cdots + x_n \leq |S_{\geq i}| + 1, \]
    for all $i \in [n]$.
\end{corollary}

\begin{proof}
    By Proposition \ref{prop: subset description for linked toric intervals}, we have that $S \subseteq [n-1]$ and $T = S \cup \{ n \}$. This gives $|T_{\geq i}| = |S_{\geq i}| + 1$ which proves the corollary by Theorem \ref{mainthm: hyperplane description}.
\end{proof}

We can also interpret these results in terms of symmetric lattice paths.

\begin{proposition}\label{prop: toric if and only if snake path}
    Let $p \leq q$ be two symmetric lattice paths. Then $[p,q]$ corresponds to a linked toric interval if and only if it is a linked symmetric snake path.
\end{proposition}

\begin{proof}

    Say that $p = \alpha_1 \cdots \alpha_{2n}$ and $q = \beta_1 \cdots \beta_{2n}$. Let $S$ and $T$ be the corresponding subsets of $[n]$ in the delta model. For any subset $A \subseteq [n]$ with corresponding symmetric lattice path $r$, a direct calculation shows that $|A_{\geq i}|$ is the integer $a$ such that $r$ intersects the line $y = n + i -1 - x$ at the coordinates $(n-a,a + i - 1)$. Therefore, the difference $|T_{\geq i}| - |S_{\geq i}|$ is equal to the number of boxes in the skew-diagram $\lambda(q)/\lambda(p)$ whose interior intersects the line $y = n +i - 1 - x$.

    Suppose that $[p,q]$ is a linked snake path. By linkedness, the paths $p$ and $q$ only intersect at $(0,0)$ and $(n,n)$. Therefore, there is at least one box in $\lambda(q) / \lambda(p)$ whose interior intersects $y = n + i -1 -x$ for all $i \in [n]$. Since this is a snake path, this Young diagram does not contain a $2\times2$ box. Therefore, the number of boxes whose interior intersects $y = n + i - 1 -x$ is at most one. Combining these two, we have that $|T_{\geq i}| - |S_{\geq i}| = 1$ for all $i \in [n]$. It follows that $T = S \cup \{n \}$. By Proposition \ref{prop: subset description for linked toric intervals}, we see that $[S,T]$ is a toric interval.

    Conversely, if $[S,T]$ is a toric interval, then $|T_{\geq i}| - |S_{\geq i}| = 1$ for all $i \in [n]$, since $T = S \cup \{n\}$ by Proposition \ref{prop: subset description for linked toric intervals}. This implies that there is exactly one box in the skew diagram of $p$ and $q$ whose interior intersects $y = n +i - 1 -x$ for all $i \in [n]$. Thus $[p,q]$ is a linked symmetric snake path.
\end{proof}

We will need the following faces of the polytopes of delta matroids corresponding to symmetric snake paths.

\begin{proposition}\label{prop: faces of snake path lpdm}
    Let $[S,T] = [S,S \cup \{n\}]$ be a linked toric interval. Let $[A,B]$ be a subinterval. Then the polytope $P(\Delta[A,B])$ is a face of $P(\Delta[S,T])$.
\end{proposition}

\begin{proof}
    

    Let $[A, B] \subsetneq [S, T] = [S, S \cup \{n\}]$, and let $J = \{j \in [n] : |A_{\geq j}| = |B_{\geq j}|\}$. Then by applying Lemma \ref{lem: Gale order and truncated sets} to the sequence of relations $S \leq A \leq B \leq S \cup \{n\}$, we have $|S_{\geq i}| \leq |A_{\geq i}| \leq |B_{\geq i}| \leq |S_{\geq i}| + 1$ for any $i \in [n]$. Note that if $|A_{\geq i}| < |B_{\geq i}|$ for all $i \in [n]$, we must have $|S_{\geq i}| = |A_{\geq i}|$ and $|B_{\geq i}| = |S_{\geq i}| + 1$ for every $i \in [n]$. This implies $A = S$ and $B = T$, which violates our assumption. Thus $J \neq \emptyset$. Moreover, we observe that $|S_{\geq i}| = |A_{\geq i}| < |B_{\geq i}| = |S_{\geq i}| + 1$ for $i \not\in J$.
    
    We can then conclude for every subinterval $[A, B]$ of the linked toric interval $[S, T]$, $P(\Delta[A, B])$ can be obtained by modifying some inequalities for $P(\Delta[S, T])$, as described in Corollary \ref{cor: hyperplane description of toric intervals}, to equalities. In particular, by Theorem \ref{mainthm: hyperplane description}, $x \in P(\Delta[A, B])$ if and only if $0 \leq x_i \leq 1$ for all $i \in [n]$ and
    \[
    \begin{cases}
        |S_{\geq i}| \leq x_i + \cdots + x_n \leq |S_{\geq i}| + 1 & \text{if } i \not\in J \\
        x_i + \cdots + x_n = |A_{\geq i}| & \text{if } i \in J 
    \end{cases}
    \]
    Since $|A_{\geq i}| \in \{|S_{\geq i}|, |S_{\geq i}| + 1\}$, we know that $P(\Delta[A, B])$ is a face of $P(\Delta[S, T])$.

\end{proof}

\subsection{Triangulations of Snake Path Delta Matroid Polytopes}

Recall that a \textbf{unimodular triangulation} of a polytope $P$ is a collection of lattice simplices $\nabla_1,\ldots,\nabla_k$ where $\dim(\nabla_i) = \dim(P)$ such that the union of the simplices is $P$, the intersection of any two $\nabla_j$ and $\nabla_j$ is a common face of both or empty, and each simplex has volume $1/d!$ where $d$ is its dimension.

There are many natural unimodular triangulations of the hypercube $[0,1]^n$. The most straight-forward unimodular triangulation is given by the collection of $n!$ simplices
\[ \nabla_{w} = \{ \left(x_{1}, \dots, x_{n} \right) \in \mathbb{R}^{n} : 0 \leq x_{w\left(1\right)} \leq  \ldots \leq x_{w\left(n\right)} \leq 1, \sum_{i=1}^{n}x_i = 1 \}\]
for $w \in S_n$.

In his study of Eulerian numbers, Stanley defined a second unimodular triangulation that will be more useful to us \cite{Sta77}. Consider the function $\psi:[0,1]^{n} \to [0,1]^{n}$ defined coordinate wise as:

\[ \psi_{i}\left(x_{1}, \dots, x_{n} \right) =x_{1} + \dots + x_{i} - \lfloor x_{1} + \dots + x_{i} \rfloor. \]
This map is piecewise-linear, bijective (outside of a set of measure $0$), and volume-preserving. The inverse map $\phi^{-1}$ is well-defined on the union of the interiors of the simplices and so it transforms the standard triangulation above into a new triangulation. Recall that the hypersimplex $\Delta_{k,n}$ is the polytope
 \[\Delta_{k,n} = \operatorname{conv}(x \; \lvert \; \sum_{i=1}^n x_i = k, 0 \leq x_i \leq 1), \]
and that it is given as $[0,1]^n \cap \{x \in \RR^n \; \lvert \; \sum_{i=1}^n x_i = k\}$.

\begin{theorem}[Stanley]\label{thm: Stanley's triangulation}

    The function $\psi^{-1}$ transforms the triangulation of the hypercube $[0,1]^{n}$ by $\nabla_{w}$'s into a unimodular triangulation such that intersecting each simplex by $\{x \in \RR^n \; \lvert \; \sum_{i=1}^{n} x_i = k\}$ gives a unimodular triangulation of $\Delta_{k,n}$.
\end{theorem}

When restricted to the simplex $\nabla_{w}$, the map $\psi^{-1} \left(x_{1}, \dots, x_{i} \right)=\left(y_{1}, \dots, y_{n} \right)$ is given by $x_{1} = y_{1}$ and

\[y_{i+1} = \begin{cases}

    x_{i + 1} - x_{i} & \text{if  } w^{-1}\left(i+1\right) > w^{-1}\left(i\right) \\ 
    x_{i+1} - x_{i} + 1 & \text{if } w^{-1}\left(i+1\right) < w^{-1}\left(i\right) 

\end{cases}\]

For our purposes, we will consider the triangulation given by the map $\phi = r \circ \psi^{-1}$ where $r: \RR^n \to \RR^n$ is the linear transformation $x_{i} \mapsto x_{n-i+1}$. For the following result, let $\operatorname{asc}(w)$ be the ascents of $w$ and $\operatorname{des}(w)$ be its descents.

\begin{proposition}\label{prop: ascent description of simplicies}
Let $[S,T] = [S, S\cup\{n\}]$ be a toric interval.  Then, $\phi(\nabla_{w}) \subseteq P(\Delta[S,S \cup \{n\}])$ if and only if $S$ is the ascent set of $w$. 
\end{proposition}

\begin{proof} The function $\phi(x_1,\ldots,x_n) = (y_1,\ldots,y_n)$ is defined explicitly within each hypersimplex $\nabla_{w}$ by $y_{n} = x_{1}$ and

\[y_{n - i} = \begin{cases}

    x_{i + 1} - x_{i} & \text{if  } w^{-1}\left(i+1\right) > w^{-1}\left(i\right) \\ 
    x_{i+1} - x_{i} + 1 & \text{if } w^{-1}\left(i+1\right) < w^{-1}\left(i\right) .

\end{cases}\]

Let $\operatorname{des}(w)_i$ be the number of descents of $w$ less than or equal to $i$. Then, the sum of the last $n - i + 1$ coordinates of $y$ is 

\begin{align*}
\sum_{j = i}^{n} y_{j} &= x_{1} + \sum_{j = 1}^{n - i} \left( x_{j+1} - x_{j} \right) +\operatorname{des}(w^{-1})_{n - i + 1}. \\
&= x_{n-i} + \operatorname{des}(w^{-1})_{n-i+1}.
\end{align*}
Hence, the image $\phi(\nabla_w)$ is contained in the polytope determined by the inequalities
\[ \text{des}(w^{-1})_{n+1-i} \leq x_i + \dots + x_n \leq \text{des}(w^{-1})_{n+1-i} + 1,\]
for all $i \in [n]$.

Let $S$ be the set of ascents of $w$. Then, $\operatorname{des}(w^{-1})_{n+1-i} = |S_{\geq i}|$. The inequalities above become
 \[|S_{\geq i}| \leq x_i + \cdots + x_n \leq |S_{\geq i}| + 1. \]
By Corollary \ref{cor: hyperplane description of toric intervals}, we see that $\phi(\nabla_w) \subseteq P(\Delta[S,S \cup \{n\}])$. Since the intersection of the polytopes of lattice path delta matroids corresponding to two different symmetric snake paths is not full-dimensional by Lemma \ref{lem: dimension of intersection of snake paths lpdm}, we have that each $\phi(\nabla_w)$ is contained in exactly one $P(\Delta[S,S \cup \{n\}])$. Therefore, $\phi(\nabla_w) \subseteq P(\Delta[S,S\cup\{n\}])$ if and only if $\operatorname{asc}(w) = S$.

\end{proof}

As a consequence, we obtain one of our main results:
\begin{theorem}[Theorem \ref{mainthm: LPDM subdivision into snake paths}]\label{thm: Triangulation of Toric Polytopes } 
    Let $[S, T]$ be a linked toric interval. Then the map $\phi$ induces an unimodular triangulation of $P(\Delta[S,T])$.    
\end{theorem}

\begin{proof}
    Since the unimodular simplices $\phi(\nabla_w)$ triangulate the hypercube $[0,1]^n$ and $P(\Delta[S,S\cup\{n\}]) \subseteq [0,1]^n$, we have that each point $x$ in the interior of $P(\Delta[S,T])$ is contained in some simplex $\phi(\nabla_w)$. 

    By Lemma \ref{lem: dimension of intersection of snake paths lpdm}, the interiors of polytopes corresponding to symmetric snake paths do not intersect. Hence, the simplex $\phi(\nabla_w)$ is not contained in any $P(\Delta[p,q])$ except possibly $P(\Delta[S,T])$. By the previous proposition, we see that each $\phi(\nabla_w)$ is contained in at least one $P(\Delta[p,q])$. This implies that the union of the $\phi(\nabla_w)$ contained in $P(\Delta[S,T])$ is the entire polytope.
\end{proof}

\begin{corollary}\label{cor: volume of toric intervals ascents}
    The volume of $P(\Delta[S,S \cup \{n\}])$ is $\frac{1}{n!}$ times the number of permutations with ascent set $S$.
\end{corollary}

\subsection{Relation to Bidkhori's Triangulation}

Our Theorem \ref{mainthm: snake path triangulation} is Lagrangian matroid generalization of the following result of Bidkhori.

\begin{theorem}\cite{Bid12}
    Let $[p,q]$ be a linked snake path with lattice path matroid $M= M[p,q]$. Then, $P(M)$ has a triangulation where each maximal dimensional simplex has volume $\frac{1}{(n-1)!}$.
\end{theorem}

In fact, we can recover her result by considering the homogeneous components of the lattice path delta matroids of symmetric snake paths.

\begin{proof}
    Let $k$ be the rank of $M$ and $n$ be the size of the ground set. The snake path $[p,q]$ corresponds to a skew-diagram with no $2 \times 2$ boxes contained in the box with southeast corner $(0,0)$ and northeast corner $(n-k,k)$

    Translate the skew-diagram so that the southwest corner of the containing box is $(k,n-k)$ and the northeast corner is $(n,n)$. Then reflect the skew-diagram along the line $y = n - x$. This gives a symmetric skew-diagram which corresponds to a (non-linked) symmetric snake path $[p',q']$ such that $p'$ and $q'$ meet exactly at $(0,0), (k,n-k)$, and $(n,n)$. Let $[r,s]$ be a linked symmetric snake path whose skew-diagram contains $[p',q']$.
    
    By the lattice path description of homogeneous components from Proposition \ref{prop: truncation in terms of lattice paths}, we have that the $k$-th homogeneous component of $\Delta[r,s]$ is exactly $M[p,q]$. Combining this with Proposition \ref{prop: truncation hyperplane description}, we have that
     \[P(\Delta[r,s]) \cap \{x \in \RR^n \; \lvert \; \sum_{i=1}^n x_i = k\} = P(M[p,q]). \]
    By Theorem \ref{thm: Stanley's triangulation}, the triangulation of $P(\Delta[r,s])$ induces a triangulation of $P(M[p,q])$ where each maximal dimensional piece is a facet of the simplices from Stanley's triangulation and, hence has volume $\frac{1}{(n-1)!}$.
\end{proof}

\subsection{Maximal Chains and Volumes}

If $X_P$ is the toric variety corresponding to the polytope $P$, then the normalized volume of the lattice polytope $P$ is equal to the degree of the variety $X_P$. Since Richardson varieties of toric intervals $X_{S}^T$ are the toric varieties of the corresponding polytope $P(\Delta[S,T])$, one can calculate the volume from the degree of the variety. In the special case where $S = \emptyset$, we have that $X_S^T$ is a Schubert variety, for which the degree has been heavily studied. In particular, there is the following standard result:

\begin{theorem}
    The degree of the Schubert variety $\Omega^w \subseteq \operatorname{LGr}_{n,2n}$ is the number of maximal chains in the poset $W/W_i$ from $e$ to $w$.
\end{theorem}

With a little work, one can prove similar results for Richardson varieties. We now verify that our unimodular triangulation implies the same result for toric intervals.

\begin{proposition}
 \label{cor: volume calculation}
    Let $[S,T] = [S, S \cup \{n\}]$ be a toric interval. The volume of $P(\Delta[S, T])$ is the number of maximal chains in the interval $[S, T]$ divided by $n!$. 
    
\end{proposition}

\begin{proof} By Corollary \ref{cor: volume of toric intervals ascents}, it suffices to show that the number of permutations with ascents in $S$ is in bijection with the maximal chains in $[S, S \cup{n}]$. We will actually construct a bijection of maximal chains with permutations with descents in $S$.

By Lemma \ref{lem: Gale order and truncated sets}, we have that $A \leq B$ if and only if $|A_{\geq i}| \leq |B_{\geq i}|$. From the description of the rank function in Lemma \ref{lem: rank calculation of Gale order on subsets}, we have that $A \lessdot B$ if and only if $B = A \backslash \ell \cup \{\ell + 1 \}$ for some $\ell \in A$. Consider a maximal chain
 \[S_0 = S \lessdot S_1 \lessdot \cdots \lessdot S_n = T \]
in $[S,T]$. 

For any integer $\ell\not = 1 \in [n]$, there is at most one index $i$ such that $\ell \in S_i$ and $S_{i+1} = S_i \backslash \ell \cup \{\ell +1\}$ and for $\ell =1$, there is at most one index where $1 \not \in S_i$ and $S_{i+1} = S_i \cup \{1\}$. Otherwise, we would have $|T_{\geq \ell}| \geq |S_{\geq \ell}| + 2$ which contradicts $T = S \cup \{n\}$. Since this maximal chain has length $n$, we see that for each integer $\ell$ there is exactly one index $c_{\ell}$ that satisfies that property.

Consider the sequence $(a_1,\ldots, a_n)$ where $a_i = c_{i}$. Let $\pi$ be the permutation with $\pi(i) = a_i$. It is clear that each maximal chain is determined uniquely by this sequence. Therefore, we have defined an injective map from maximal chains of $[S,T]$ to $S_n$.

We now show that $\pi$ has descent set $S$. Let $i \in S$, then we want to show that $a_i > a_{i+1}$. At each step of the maximal chain, a unique element increases by $1$ or the element $1$ is added. We refer to the step $S_{i} = S_{i-1} \backslash \ell \cup \{\ell+1\}$ as $\ell$ moving to $\ell + 1$. The integer $\ell-1$ moves to $\ell$ exactly at step $a_{\ell}$ and $1$ is added at step $a_1$. Notice that since $i \in S$, we have that $i-1$ cannot move to $i$ until $i$ moves to $i+1$. Therefore, the step $a_{i+1}$ where $i$ moves must happen before the step $a_{i}$ where $i-1$ moves. If $i$ is $1$, then $1$ cannot be added until the first $1$ moves. This implies that $a_i > a_{i+1}$ for $i \in S$.

Likewise, if $i \not \in S$, then $i$ cannot move into $i+1$ until $i-1$ moves into $i$. Therefore, $a_{i} < a_{i+1}$. All together, this shows that the map from maximal chains to permutations has image contained in the permutations with descents in $S$.

Conversely, if $\pi$ is a permutation with descents in $S$, let $(a_1, \ldots, a_n)$ be the corresponding sequence where $a_i = \pi(i)$ and define $b_i = \pi^{-1}(i)$. Then, define the subsets $S_0 = S$ and
 \[S_i = S_{i-1} \backslash \{b_i-1\} \cup \{b_i\}\]
if $b_i \not = 1$ and
 \[S_i = S_{i-1} \cup \{1\}, \]
otherwise, for $i = 1, n$. So, $a_\ell$ represents the step at which integer $\ell-1$ moved to $\ell$ and $b_i$ is the element $\ell$ such that $\ell-1$ moved to $\ell$ at step $i$. For this to be useful, we need to show that $b_k -1 \in S_k$ for $b_k \not = 1$ and $b_k \not \in S_k$. The case where $b_k = 1$ is immediate since position $b_1 = \pi^{-1}(1)$ is never a descent of $\pi$ and so it is not in $S = S_0$.

Now we show that $b_k \not \in S_k$ for the remaining values. If position $b_{k}$ is a descent in $\pi$, then we know that position $b_{k} + 1$ has a smaller value and $b_{k} \in S_0$. In other words, $a_{b_k} > a_{b_k + 1}$. This means that the integer $b_k$ moved to $b_k + 1$ at an earlier step which implies that $b_k \not \in S_{k-1}$.  If position $b_k$ is not a descent in $\pi$, then we know that $b_k \not \in S_k$ and since the move at step $k$ is the unique move where $k-1$ moves to $k$.

For the inclusion $b_k -1 \in S_{k-1}$, if position $b_k - 1$ is a descent then $b_k -1 \in S_0$ and since step $k$ is the only move where $b_k -1$ moves, we have that $b_k -1 \in S_{k-1}$. If $b_k - 1$ is not a descent, then $a_{b_k - 1} < a_{b_k}$ which implies that $b_k -2$ moves to $b_k -1$ before step $k$. Therefore, $b_k -1 \in S_{k-1}$.

We finally show that the sequence $S_0, S_1, \ldots, S_n$ actually corresponds to a chain in the toric interval $[S,T]$. Since each element $\ell \in [n]$ is added exactly once in this sequence, we have that $|({S_n})_{\geq i}| = |S_{\geq i}| + 1$ which implies that $S_n = S \cup \{n\} = T$. Further, it is clear from definitions that $S_i \lessdot S_{i+1}$ which implies that this is a maximal chain. These two constructions are inverses and so we have demonstrated a bijection.

\end{proof}

\begin{example} We give examples of the bijections constructed in the previous proof for $S = \{135\}$ and $n = 6$. In one direction, the chain
 \[135 \lessdot 136 \lessdot 236 \lessdot 1236 \lessdot 1246 \lessdot 1346 \lessdot 1356\]
corresponds to the permutation $\pi$ given in one-line notation by $(3,2,5,4,6,1)$. The descents of this permutation are at positions $1,3,$ and $5$.

In the other direction, given the permutation $\pi = (5,3,6,1,4,2)$, we have that $\pi^{-1} = (4,6,2,5,1,3)$ and the corresponding chain is
 \[135 \lessdot 1\textbf{4}5 \lessdot 14\textbf{6} \lessdot \textbf{2}46 \lessdot 2\textbf{5}6 \lessdot \textbf{1}256 \lessdot 1\textbf{3}56.\]
\end{example}

\section{Decomposition of Lattice Path Delta Matroid Polytopes}\label{Sec: Decompositions LPDM}

We now turn to studying the polytopes of general lattice path delta matroids. Our main strategy will be to decompose these polytopes into polytopes corresponding to symmetric snake paths lattice path matroids.

\subsection{Subdivision into Toric Intervals}

\begin{theorem} Let $\Delta[S,T]$ be a linked lattice path delta matroid. Then, $P = P(\Delta[S,T])$ has a polytope decomposition where each maximal polytope is the feasible polytope of the delta matroid of a linked symmetric snake path.

Explicitly,
\[ P(\Delta[S,T])  = \bigcup_{S \leq R \leq T \backslash \{n\}} P(\Delta[R, R\cup\{n\}]). \] 

\end{theorem}

\begin{proof}

By Theorem \ref{thm: hyperplane description}, the polytope $P$ is defined by the inequalities

\[ |S_{\geq i}| \leq x_{\geq i} \leq |T_{\geq i}|  \quad \text{for } 1 \leq i \leq n. \]
First, we show that for each $S \leq R \leq T \backslash \{n\}$, the polytope $P(\Delta[R,R\cup \{n\}])$ is contained in $P$. 

Since $R \leq T \backslash \{n\}$, we have that $|R_{\geq i}| \leq | (T \backslash \{n\})_{\geq i}| < |T_{\geq i}|$. Therefore, $|S_{\geq i}| \leq |R_{\geq i}| < |T_{\geq i}|$ for all $i \in [n]$. If $x$ is in $P(\Delta[R,R \cup \{n\}])$ ,which by Corollary \ref{cor: hyperplane description of toric intervals} is the same as $x$ satisfying
 \[|R_{\geq i}| \leq x_i + \cdots x_n \leq |R_{\geq i}| + 1,\]
then $x$ satisfies the inequalities defining $P$.

Now, we show that these polytopes cover $P$. Let $x \in P$ and for each $i\in [n]$ let $a_i$ be the largest integer and $b_i$ the smallest integer such that $a_i \leq x_i + \cdots x_n \leq b_i$. It is clear that $a_i \geq a_{i+1}$ and $b_i \geq b_{i+1}$ since $x_i > 0$ for all $i$. The condition $x_i \leq 1$ implies that $a_i$ is either equal to $a_{i+1}$ or $a_{i+1} + 1$ and $b_i$ is either $b_{i+1}$ or $b_{i+1} + 1$. Further, we have that $b_i - a_i$ is either $0$ or $1$.

Let $A$ and $B$ be the subsets defined by $i \in A$ and $i \in B$ if and only if $a_i = a_{i+1} + 1$ $b_i = b_{i+1} + 1$, respectively. By the construction of $a$ and $b$, we see that $|A_{\geq i}| = a_i$ and $|B_{\geq i}| = b_i$. The interval $[A,B]$ corresponds to a (possibly not-linked) symmetric snake path. Our construction ensures that $x \in P(\Delta[A,B])$.

Since $P(\Delta[S,T])$ is full-dimensional, we know by Proposition \ref{prop: dimension formula}, that the corresponding symmetric lattice paths only intersect at $(0,0)$ and $(n,n)$. This implies that there is a linked symmetric snake path contained in the skew diagram of $[S,T]$ that contains the symmetric snake path corresponding to $[A,B]$. This linked symmetric snake path gives a toric interval $[R,R \subseteq \cup \{n\}]$ such that
\[S \leq R \quad \quad \text{and} \quad \quad R \cup \{n \} \leq T \iff R \leq T \backslash \{n\}. \]
We have shown that $x \in P(\Delta[A,B]) \subseteq P(\Delta[R,R \cup \{n\}])$.

All that remains to show is that the polytopes $P(\Delta[R,R \cup \{n\}])$ form a polytope decomposition. By Proposition \ref{prop: intersections of LPDM are LPDM}, the intersection of any two of these polytopes $P_1$ and $P_2$ is again a lattice path delta matroid $\Delta[A,B]$. Further, the skew-diagram of $[A,B]$ is obtained by intersecting the skew-diagrams of the two polytopes we are intersecting. This implies that $[A,B]$ is a symmetric snake path whose skew-diagram is contained in the skew-diagrams of the two toric intervals. This implies that $[A,B]$ is a subinterval of $[S,T]$ and then Proposition \ref{prop: faces of snake path lpdm} implies that $P(\Delta[A,B])$ is a face of both $P_1$ and $P_2$.

\end{proof}

\begin{corollary}
    Let $\Delta[S,T]$ be a linked lattice path delta matroid. Let $m$ be the number of permutations with ascent set $R$ such that $S \leq R \leq T \backslash \{n\}$. Then,
     \[\operatorname{vol}(P(\Delta[S,T])) = \frac{m}{n!}. \]
\end{corollary}

\begin{proof}
    This follows from the combination of Theorem \ref{mainthm: LPDM subdivision into snake paths} and Corollary \ref{cor: volume calculation}.
\end{proof}

\subsection{Relation to Bidkhori's Decomposition}

Our work in this section is a type $C_n$ generalization of the following result of Bidkhori:

\begin{theorem}\cite{Bid12}
    Let $[S,T]$ be subsets of $[n]$ of the same cardinality $k$ such that $P(M[S,T])$ has dimension $n-1$. Then, the lattice path matroid polytope $P(M[S,T])$ has a matroid polytope subdivision into lattice path matroid polytopes corresponding to linked snake paths in the $k \times (n-k)$ box.
\end{theorem}

In fact, our work recovers her result.

\begin{proof}
    Let $\mu$ be the skew-diagram corresponding to the subsets $S$ and $T$ in the $k \times (n-k)$ box. Place the box so that the southwest corner is at coordinate $(n-k,k)$ and the northeast corner is at $(n,n)$. Reflecting this diagram along the line $y = n- x$ we get a symmetric skew-diagram in the box $[0,1]^n$. This is the skew-diagram of lattice paths $[p',q']$ such that $p'$ and $q'$ meet exactly at $(0,0)$, $(n-k,k)$, and $(n,n)$. Let $[p,q]$ be a linked interval of symmetric lattice paths whose skew-diagram contains the skew-diagram of $[p',q']$.

    By Proposition \ref{prop: truncation in terms of lattice paths}, we see that the $k$th truncation of $\Delta[p,q]$ is the matroid $M[S,T]$. Therefore,
     \[P(\Delta[p,q]) \cap \{ x \in \RR^n \; \lvert \; \sum_{i=1}^{n}x_i = k\} = P(M[S,T]). \]

    Now, let $[A,B]$ be the subsets corresponding to a linked snake path in the skew diagram $\mu$. By reflecting along the line $y = n-x$, we obtain a symmetric snake path. Since $[p,q]$ is linked, there is a symmetric snake path $[S', S' \cup \{n\}]$ whose skew-diagram contains the reflected skew-diagram of $[A,B]$. Therefore, the $k$th homogeneous component of $\Delta[S', S' \cup \{n\}]$ is the matroid $M[A,B]$. 
    
    Note that the homogeneous components of a delta matroid corresponding to a symmetric snake path is the matroid corresponding to a snake path. All of this implies that intersecting the polytope decomposition of $P(\Delta[p,q])$ from Theorem \ref{mainthm: LPDM subdivision into snake paths} with the hyperplane $\{x \in \RR^n \; \vert \; \sum x_i = k\}$ gives a polytope decomposition of $P(M[S,T])$ into the polytopes of snake paths. Further, every snake path appears in this way.
\end{proof}

\bibliographystyle{plain} 
\bibliography{main.bib} 

\begin{thebibliography}{10}

\bibitem{Ardila2003}
Federico Ardila.
\newblock The {C}atalan matroid.
\newblock {\em Journal of Combinatorial Theory. Series A}, 104, 2003.

\bibitem{Ardila2023}
Federico Ardila and Mario Sanchez.
\newblock Valuations and the {Hopf} monoid of generalized permutahedra.
\newblock {\em International Mathematics Research Notices}, 2023.

\bibitem{BK22}
Carolina Benedetti and Kolja Knauer.
\newblock Lattice path matroids and quotients.
\newblock {\em arXiv preprint arXiv:2202.11634}, 2022.

\bibitem{Bid12}
Hoda Bidkhori.
\newblock Lattice path matroid polytopes.
\newblock {\em arXiv preprint arXiv:1212.5705}, 2012.

\bibitem{BB05}
Anders Bj{\"o}rner and Francesco Brenti.
\newblock {\em Combinatorics of {C}oxeter groups}, volume 231.
\newblock Springer, 2005.

\bibitem{Bonin2003}
Joseph Bonin, Anna de~Mier, and Marc Noy.
\newblock Lattice path matroids: Enumerative aspects and {T}utte polynomials.
\newblock {\em Journal of Combinatorial Theory. Series A}, 104, 2003.

\bibitem{Bonin2010}
Joseph~E. Bonin.
\newblock Lattice path matroids: The excluded minors.
\newblock {\em Journal of Combinatorial Theory. Series B}, 100, 2010.

\bibitem{Bonin2006}
Joseph~E. Bonin and Anna de~Mier.
\newblock Lattice path matroids: Structural properties.
\newblock {\em European Journal of Combinatorics}, 27, 2006.

\bibitem{Borovik2003}
Alexandre~V. Borovik, Izrail'~Moiseevich Gel'fand, Neil White, Alexandre~V Borovik, I.M. Gelfand, and Neil White.
\newblock {\em Coxeter matroids}.
\newblock Springer, 2003.

\bibitem{EFLS22}
Christopher Eur, Alex Fink, Matt Larson, and Hunter Spink.
\newblock Signed permutohedra, delta-matroids, and beyond.
\newblock {\em arXiv preprint arXiv:2209.06752}, 2022.

\bibitem{Eur2021}
Christopher Eur, Mario Sanchez, and Mariel Supina.
\newblock The universal valuation of {C}oxeter matroids.
\newblock {\em Bulletin of the London Mathematical Society}, 53, 2021.

\bibitem{ferroni2022valuative}
Luis Ferroni and Benjamin Schr{\"o}ter.
\newblock Valuative invariants for large classes of matroids.
\newblock {\em arXiv preprint arXiv:2208.04893}, 2022.

\bibitem{knauer2018lattice}
Kolja Knauer, Leonardo Mart{\'\i}nez-Sandoval, and Jorge~Luis Ram{\'\i}rez~Alfons{\'\i}n.
\newblock On lattice path matroid polytopes: integer points and {E}hrhart polynomial.
\newblock {\em Discrete \& Computational Geometry}, 60:698--719, 2018.

\bibitem{Oxley}
J.~Oxley.
\newblock {\em Matroid Theory}.
\newblock Oxford Graduate Texts in Mathematics. OUP Oxford, 2011.

\bibitem{Reiner1997}
Victor Reiner.
\newblock Non-crossing partitions for classical reflection groups.
\newblock {\em Discrete Mathematics}, 177, 1997.

\bibitem{Sta77}
Richard Stanley.
\newblock Eulerian partitions of a unit hypercube.
\newblock {\em Higher Combinatorics}, 31:49, 1977.

\bibitem{StumpDyck}
Christian Stump.
\newblock More bijective {C}atalan combinatorics on permutations and on signed permutations.
\newblock {\em Journal of Combinatorics}, 4, 08 2008.

\bibitem{Tsukerman2015}
E.~Tsukerman and L.~Williams.
\newblock Bruhat interval polytopes.
\newblock {\em Advances in Mathematics}, 285, 2015.

\end{thebibliography}

\end{document}